\documentclass[11pt,reqno]{amsart}

\usepackage[T1]{fontenc}
\usepackage[utf8]{inputenc}
\usepackage{lmodern}
\usepackage{physics}
\usepackage{geometry}
\usepackage{xcolor}
\usepackage{amsmath,amssymb,amsthm,mathtools,bm}
\usepackage{mathrsfs}

\usepackage[hidelinks]{hyperref}
\hypersetup{
    colorlinks=false,
    pdfborder={0 0 0}
}

\numberwithin{equation}{section}

\newtheorem{theorem}{Theorem}[section]
\newtheorem{proposition}[theorem]{Proposition}
\newtheorem{lemma}[theorem]{Lemma}
\newtheorem{sublemma}[theorem]{Sublemma}
\newtheorem{corollary}[theorem]{Corollary}

\newtheorem*{property*}{Property}

\theoremstyle{definition}
\newtheorem{definition}[theorem]{Definition}

\newtheorem{example}[theorem]{Example}
\theoremstyle{remark}
\newtheorem{remark}[theorem]{Remark}

\newcommand{\cone}{\mathcal{C}}
\newcommand{\indic}{\mathbf{1}}
\newcommand{\paraop}{\mathscr{D}}
\newcommand{\R}{\mathbb R}
\newcommand{\E}{\mathbb E}

\newcommand{\sfp}{\mathsf{p}}
\newcommand{\sfq}{\mathsf{q}}

\newcommand{\Opendenseset}{\mathcal{U}}

\newcommand{\PathSpace}{\mathcal Q}
\newcommand{\SmoothCore}{\mathcal S}
\newcommand{\Domain}{\Omega}
\newcommand{\LinearOp}{\mathscr L}

\newcommand{\SigN}{\Sigma_N}

\newcommand{\Hopf}{\mathscr H}

\newcommand{\xionesp}{\xi^{\mathrm{1sp}}}

\newcommand{\bmlambda}{\bm{\lambda}}
\newcommand{\bmDelta}{\bm{\Delta}}
\newcommand{\bmI}{\bm{I}}
\newcommand{\bmR}{\bm{R}}

\newcommand{\bmp}{\bm{p}}

\newcommand{\bmx}{\bm{x}}

\newcommand{\Residual}{\mathcal{R}}

\newcommand{\Hoelder}[1]{C^{#1}}
\newcommand{\ParaHoelder}[1]{C_{b}^{1+{#1}/2,2+{#1}}}
\newcommand{\ParaHoelderzero}[1]{C_b^{{#1}/2,{#1}}}

\title{A Concavity Theorem for the Parisi PDE}
\author{Fu-Hsuan Ho}
\address{Department of Mathematics, Weizmann Institute of Science, Israel}
\email{fu-hsuan.ho@weizmann.ac.il}

\begin{document}

\begin{abstract}
We prove that the map sending the diffusion profile to the solution of 
a time-changed Parisi PDE evaluated at time-space $(0,0)$ is concave.
This result strengthens the raywise concavity result proven by Auffinger and Chen (2016).
As an application, for the balanced multispecies Ising spin glasses, the lower bound of Bates and Sohn (2025)
matches the Hopf-type upper bound given by the Hamilton--Jacobi framework developed by Mourrat, Chen and Xia.
\end{abstract}


\maketitle

\section{Introduction}
 
In the classical Sherrington--Kirkpatrick (SK) model, 
the celebrated work of Parisi \cite{Parisi1979, Parisi1980Sequence}
predicted that its free energy is given by the so-called Parisi formula.
Guerra \cite{Guerra2003BrokenRSB} and Talagrand \cite{Talagrand2006ParisiFormula}
established Parisi's prediction rigorously, and Panchenko \cite{Panchenko2014MixedPSpin} extended the Parisi formula to mixed $p$-spin models.

Multispecies spin glasses generalize the SK model by partitioning the spins into \(D\) subsets, called species.
Mathematically, the model is defined as follows.

\begin{definition}[Multispecies mixed Ising model]
\label{def:multispecies}
Let $N\geq 1$, $D\geq 1$ and $\Sigma_N=\{\pm 1\}^N$.
For all $\bmx\in\mathbb R^D$, define the covariance function
\[
    \xi(\bmx)
    =
    \sum_{k\geq 1}
    \xi_k(\bmx)
    =
    \sum_{k\geq 1}
    \sum_{\substack{\bm p\in\mathsf P_k}}
    \Delta_{\bmp}^2\prod_{d=1}^D x_d^{p_d}.
\]
where $\bmDelta=(\Delta_{\bmp}^2:\bmp\in\mathsf P_k,\, k\geq 1)$ is a family of nonnegative coefficients and
where
\[
    \mathsf P_k
    =
    \{\bmp=(p_1,\ldots,p_D)\in\mathbb Z_+^D:
      |\bmp|:=p_1+\cdots+p_D=k\}.
\]
We also assume that $\xi((1+\varepsilon)\bm{1})<\infty$ for some $\varepsilon>0$.

For \(\sigma,\tau\in\SigN\), define the species overlap vector by
\[
    R_{N,d}(\sigma,\tau)
    =
    \frac1N\sum_{i\in I_{N,d}}\sigma_i\tau_i,
    \qquad
    \bmR_N(\sigma,\tau)=(R_{N,d}(\sigma,\tau))_{d=1}^D,
\]
where \(\bmI_N=\{I_{N,1},\ldots,I_{N,D}\}\) is a partition of $\{1,\ldots,N\}$.

The Hamiltonian of the \emph{multispecies mixed Ising model} is a centered Gaussian process
 \(H_N=(H_N(\sigma))_{\sigma\in\SigN}\)
with covariance
\[
    \mathbb E H_N(\sigma)H_N(\tau)
    =
    N\xi(\bmR_N(\sigma,\tau)),
    \qquad \sigma,\tau\in\SigN.
\]
\end{definition}

Throughout the paper, we assume that the proportion of each block in the
partition is nondegenerate as $N \to \infty$. More precisely,
letting
\[
    N_{N,d}=|I_{N,d}|,
    \qquad
    \lambda_{N,d}=N_{N,d}/N,
    \qquad
    \bmlambda_N=(\lambda_{N,1},\ldots,\lambda_{N,D}),
\]
we assume that \(\bmlambda_N\to\bmlambda=(\lambda_1,\ldots,\lambda_D)\), where
\(\lambda_d>0\) for every \(d\in\{1,\ldots,D\}\) and
\(\sum_{d=1}^D\lambda_d=1\). 

\begin{remark}
When the context is clear, we simply refer to the model in Definition~\ref{def:multispecies} as the \emph{multispecies model}.
Note that a \emph{multispecies model} can be specified by $(\xi,\bmlambda_N)$.
We also note that, when $D=1$, Definition~\ref{def:multispecies} reduces to the mixed $p$-spin model, which will be referred to 
as the \emph{one-species model}. In this case $\bmlambda_N = \lambda_1=1$, 
so the model is defined by $\xi$.
\end{remark}

Given the parameters $(\xi,\bmlambda_N)$, the free energy
of the multispecies model is defined as
\begin{equation}
\label{eq:hopfapp-positive-free-energy}
    F_N(\xi,\bmlambda_N)
    =
    \frac1N\mathbb E\log
    \int_{\SigN}\exp(H_N(\sigma))\dd{P_N(\sigma)},
\end{equation}
where \(P_N=(\frac{1}{2}\delta_{-1}+\frac{1}{2}\delta_1)^{\otimes N}\) is the uniform probability measure on \(\SigN\).
When $D=1$, we denote by $F_N(\xi)$ the free energy of the one-species model with the covariance function $\xi$.

As in the one-species case, the question of interest is to compute the limiting free energy.
The first rigorous result for the limiting free energy of multispecies spin glasses was obtained by Barra, Contucci, Mingione, and Tantari \cite{MR3311887}
and Panchenko \cite{MR3433586}. However, the method of Barra et al. \cite{MR3311887} requires 
convexity of $\xi$, so it does not cover, for example, the bipartite model,
where $D=2$ and $\xi(x_1,x_2)=x_1x_2$.

To tackle nonconvex models, Mourrat and coauthors developed a Hamilton--Jacobi framework to 
study the free energy of spin-glasses with possibly nonconvex covariance.
For a pedagogical discussion of this approach, see \cite{DominguezMourrat2024HJApproach}.
In \cite{Mourrat2021NonconvexInteractions}, Mourrat proved an upper bound\footnote{Note that Mourrat and coauthors' convention of the free energy has a sign difference with the usual convention.
For consistency of the introduction, we align with the usual convention.
}
for the free energy of the bipartite model, 
and he later proved the same type of upper bound for nonconvex vector spin glasses in \cite{Mourratvector}. 

On the other hand, for the lower bound, less is known in the literature. 
When $\xi$ has permutation invariant coordinates, Issa
proved (see Theorem~1.4 in \cite{IssaPermutation}) a lower bound for the limit inferior of the free energy via a one-species model constructed from symmetrizing $\xi$.
Recently, 
Bates and Sohn \cite{BatesSohnBalanced} showed that
for the \emph{balanced} multispecies models,
the limit inferior of the free energy is lower bounded by the Parisi formula of a one-species mixed model.
\begin{definition}
    \label{def:balanced}
    Let $\xi$ be the covariance function of a multispecies model.
We say that the pair $(\xi,\bmlambda)$ is \emph{balanced} if, for every $k\ge 1$,
\begin{equation}
    \label{eq:balanced}
    \partial_{x_1}\xi_k(\bmlambda)
    =
    \partial_{x_2}\xi_k(\bmlambda)
    =
    \cdots
    =
    \partial_{x_D}\xi_k(\bmlambda).
\end{equation}
\end{definition}

\begin{remark}
\label{rem:different-conventions}
Definition~\ref{def:balanced} is equivalent to Bates--Sohn's balance condition (H3). 
When \(k=1\), Definition~\ref{def:balanced} collapses to the requirement that
\(\Delta_d^2\) is independent of \(d=1,\ldots,D\).
For $k\geq 2$, writing
\[
\xi_k(\bmx)=\sum_{d_1,\ldots,d_k=1}^D
\tilde\Delta^2_{d_1,\ldots,d_k}
x_{d_1}\cdots x_{d_k},
\]
with $\tilde\bmDelta$ the symmetrization of $\bmDelta$, one has
\[
\partial_{x_d}\xi_k(\bmlambda)
=
k\sum_{d_2,\ldots,d_k}
\Delta^2_{d,d_2,\ldots,d_k}
\lambda_{d_2}\cdots\lambda_{d_k},
\]
which is independent of $d$.
\end{remark}

We recall a few examples of the balanced models.

\begin{example}
As explained in Example 1.2(a) in \cite{BatesSohnBalanced}, the balanced SK model
satisfies Definition~\ref{def:balanced}. In general, a multispecies mixed Ising model with permutation invariant $\xi$ is balanced.
\end{example}

Our goal is to match Bates and Sohn's lower bound for the balanced models with Mourrat's upper bound.
\begin{theorem}
\label{thm:balancedFE}
Assume that $(\xi,\bmlambda)$ is balanced. Then,
\begin{equation}
    \lim_{N\rightarrow\infty} F_N(\xi,\bmlambda_N)
    =
    \lim_{N\rightarrow\infty} F_N(\xionesp)
\end{equation}
where $F_N(\xionesp)$ is the free energy of the one-species model
with the covariance function $\xionesp(x) = \sum_{k\geq 1} \xi_k(\bmlambda) x^k$.
\end{theorem}

\begin{remark}
\label{rem:BatesSohnQuestion}
Theorem~\ref{thm:balancedFE} affirmatively answers the question raised by Bates and Sohn in Remark~1.9 in \cite{BatesSohnBalanced}.
\end{remark}

To find the matching upper bound, a key input is the convexity of a
functional, which we will define in the next paragraph.
This convexity will then imply the convexity of \eqref{eq:multispecies-initial-cond},
and will allow us to apply a Hopf upper bound, see \eqref{eq:HJ1} below, under the
Hamilton--Jacobi framework.

For all $p\in [1,\infty)$, introduce the path space
\begin{equation}
    \label{eq:hopfapp-scalar-path-space}
    \PathSpace_p
    = 
    \Bigl\{\sfq\in L^p([0,1);\mathbb R_+)
    \,\Big|\,
    \sfq \text{ is c\`adl\`ag and nondecreasing}
    \Bigr\}.
\end{equation}
When \(\sfq\in\PathSpace_\infty:=\PathSpace_1\cap L^\infty\), set $\sfq(1):=\lim_{r\uparrow1}\sfq(r)$.

For $\sfq\in\PathSpace_\infty$, define the functional
\begin{equation}
    \label{eq:def-psi}
    \psi(\sfq)=\sfq(1)-\Phi^{2\sfq}(0,0),
\end{equation}
where \(\Phi^\sfq\) is the
solution of a Parisi PDE
\begin{equation}
  \label{eq:q-pde}
\begin{cases}
-\partial_\tau\Phi^\sfq(\tau,x)
=
\dfrac12\left(
\partial_{xx}\Phi^\sfq(\tau,x)
+
\sfq^{-1}(\tau)(\partial_x\Phi^\sfq(\tau,x))^2
\right),
& \tau\in [0,\sfq(1)),
\\[1ex]
\Phi^\sfq(\sfq(1),x)=\log\cosh x,
& x\in\R.
\end{cases}
\end{equation}
Here, \(\sfq^{-1}:[0,\sfq(1)]\to[0,1]\) denotes the right generalized inverse of \(\sfq\).

We now state the convexity of the functional $\psi$.
\begin{theorem}
\label{thm:chen-initial-convexity}
The functional $\psi$ admits a unique \(L^1\)-continuous extension to \(\mathcal Q_1\).
This extension, still denoted by \(\psi\), is convex.
\end{theorem}
\begin{remark}
    Theorem~\ref{thm:chen-initial-convexity} holds for any multispecies model
    introduced in Definition~\ref{def:multispecies}.
    However, even if one believes that the Hopf upper bound, see \eqref{eq:HJ1} below, is sharp beyond balanced models,
one might need new ideas to find the matching lower bound.
\end{remark}
Since \(\sfq\mapsto \sfq(1)\) is affine and \(\sfq\mapsto 2\sfq\) is linear, the convexity of
\(\psi\) is equivalent to the concavity of \(\sfq\mapsto \Phi^{\sfq}(0,0)\). 
On the other hand, from the form of the Parisi PDE \eqref{eq:q-pde}, the dependence of $\Phi^{\sfq}$ on $\sfq$ seems to be 
highly nonlinear.  
To facilitate the analysis, we thus consider a change of variables. 
Formally, on a dense class of
paths with sufficient regularity, the change of variables
\[
    \Psi^{\dot\sfq}(s,x)=\Phi^\sfq(\sfq(s),x)
\]
transforms the problem into the concavity of
\(\gamma\mapsto\Psi^\gamma(0,0)\). 
Moreover, the PDE for \(\Psi^\gamma\) is a time-changed Parisi PDE with diffusion profile \(\gamma\).
The precise formulation will be given in Section~\ref{sec:time-changed-parisi-pde} below, and we will transform
the problem into Theorem~\ref{thm:concavity}.

We remark that the concavity of $\sfq\mapsto\Phi^{\sfq}(0,0)$ should be distinguished from the strict convexity of
the map $\mu\mapsto \Phi^\mu(0,0)$ on $\mu\in\mathrm{Pr}([0,1])$,
proven by Auffinger and Chen~\cite{AuffingerChen2015}.
Here, $\mathrm{Pr}([0,1])$ denotes the set of probability measures on $[0,1]$.
Indeed, if we restrict $\PathSpace_\infty$ to the set of quantile functions on $[0,1]$, then 
$\Phi^\mu(0,0)$ can be viewed as a functional on $\mu\in\mathrm{Pr}([0,1])$.
The strict convexity of $\mu\mapsto \Phi^\mu(0,0)$ was used in \cite{AuffingerChen2015}
to establish the uniqueness of the Parisi measure. 

While the previous exposition focuses on Ising multispecies models, 
the multispecies models also admit a spherical version, where the spins are constrained to lie on the sphere.
In particular, Bates and Sohn's lower bound also holds for the balanced spherical multispecies models.
For the earlier works, Baik and Lee \cite{BaikLee2020} determined the limiting free energy of the bipartite spherical model.
Bates and Sohn \cite{BatesSohn2022FreeEnergyMultiSpeciesSpherical} proved a Parisi formula for multispecies mixed spherical models under
a convexity assumption in the upper bound.  
Subag \cite{Subag2025TAPMultiSpeciesSphericalI,Subag2023TAPMultiSpeciesSphericalII}
developed a TAP
approach for multispecies spherical models and used it to compute the limiting free
energy for pure multispecies spherical models under appropriate convergence
assumptions.
Although the present paper does not treat the spherical case, the method is expected to extend to that setting. 
This is left for future work.

\subsection*{Organization}

Section~\ref{sec:hopf-formula-application} applies Theorem~\ref{thm:chen-initial-convexity} to prove 
Theorem~\ref{thm:balancedFE}.
Section~\ref{sec:time-changed-parisi-pde} introduces the time-changed Parisi PDE and reformulates the problem into Theorem~\ref{thm:concavity}.
Then, the proof of
Theorem~\ref{thm:concavity} occupies Section~\ref{sec:pde-notation}--Section~\ref{sec:cone-preservation},
and Section~\ref{sec:heuristics} provides the proof heuristic and outline of Theorem~\ref{thm:concavity}.
Finally, Section~\ref{sec:proof-chen-initial-convexity} proves
Theorem~\ref{thm:chen-initial-convexity}. 

\section{Proof of Theorem~\ref{thm:balancedFE}}
\label{sec:hopf-formula-application}

For a balanced pair $(\xi,\bmlambda)$,
as mentioned in the introduction, Theorem~1.3 in \cite{BatesSohnBalanced}
showed that
\begin{equation}
    \liminf_{N\rightarrow\infty} F_N(\xi,\bmlambda_N)
    \geq 
    \lim_{N\rightarrow\infty} F_N(\xionesp).
    \label{eq:balanced-superior}
\end{equation} 
Therefore, it suffices to show that
\begin{equation}
    \limsup_{N\rightarrow\infty} F_N(\xi,\bmlambda_N)
    \leq 
    \lim_{N\rightarrow\infty} F_N(\xionesp).
    \label{eq:balanced-superior}
\end{equation} 

\subsection*{Inputs from the Hamilton--Jacobi framework}
We first record the two consequences of the Hamilton--Jacobi framework.
Define the Hopf formula with the parameters $(\xi,\bmlambda)$.
\begin{equation}
    \Hopf(\xi,\bmlambda)
    =
    \sup_{\sfp\in \PathSpace_\infty^D}
    \inf_{\sfq\in \PathSpace_\infty^D}
    \left\{
    \sum_{d=1}^D \lambda_d \psi(\sfq_d)
    -
    \sum_{d=1}^D \int_0^1 \sfp_d(s)\sfq_d(s)\dd{s}
    +
    \frac{1}{2}
    \int_0^1 \xi(\sfp(s)) \dd{s}
    \right\}.
\end{equation}
For a one-species model with the parameter $\xi$, we 
define
\begin{equation}
    \Hopf(\xi)
    =
    \sup_{\sfp\in \PathSpace_\infty}
    \inf_{\sfq\in \PathSpace_\infty}
    \left\{
    \psi(\sfq)
    -
    \int_0^1 \sfp(s)\sfq(s)\dd{s}
    +
    \frac{1}{2}
    \int_0^1 \xi(\sfp(s)) \dd{s}
    \right\}.
\end{equation}

With the notation above, we now state the two 
consequences where the sign and normalization are converted to the usual conventions.
\begin{enumerate}
    \item Given a multispecies model with parameters $(\xi,\bmlambda)$,
    we have 
    \begin{equation}
    \limsup_{N\rightarrow\infty} F_N(\xi,\bmlambda_N)
    \leq 
    \frac{1}{2}\xi(\bmlambda) 
    - 
    \Hopf(\xi,\bmlambda).
    \tag{HJ1}
    \label{eq:HJ1}
    \end{equation}
    The upper bound \eqref{eq:HJ1} is a specialized version of Chen--Xia's reformulation (Theorem~4.14 in \cite{ChenXiaHJ2025}) of Mourrat's upper bound
(Theorem 3.4 in \cite{Mourratvector}), combined
with their Hopf representation (Theorem~4.7(3) in \cite{ChenXiaHJ2025}). 
These results are applicable because 
we can identify the multispecies Ising model as a vector spin model 
by sending a spin $\sigma_d$ in species \(d\) to the vector \(\sigma_d e_d\in\R^D\).
Then, the vector overlap matrix is diagonal, with diagonal entries \(R_{N,d}\),
and the covariance is obtained from
\[
\hat \xi(A)=\xi(A_{11},\ldots,A_{DD}),
\]
where $A$ is a $D$-by-$D$ symmetric matrix. 
Thus, on diagonal paths, \(\hat\xi\) is just the original multispecies covariance
function.

Therefore, we can apply Theorem~4.14 in \cite{ChenXiaHJ2025} to bound the limit superior of the free energy with
the Lipschitz viscosity solution of the Hamilton--Jacobi equation stated in that paper.
Then, we apply Theorem~4.7(3) in \cite{ChenXiaHJ2025} to identify
that Lipschitz viscosity solution with their Hopf formula with $t=1/2$, $\mu=\mathsf{0}$.
Here, we replace their $\psi$ with the specialization for multispecies Ising models
\begin{equation}
    \sum_{d=1}^D \lambda_d \psi(\sfq_d),
    \label{eq:multispecies-initial-cond}
\end{equation}
so the convexity of \eqref{eq:multispecies-initial-cond} is provided by Theorem~\ref{thm:chen-initial-convexity}.
    \item Given a one-species model with the covariance function $\xi$, 
    we have
    \begin{equation}
    \lim_{N\rightarrow\infty} F_N(\xi)
    =
    \frac{1}{2}\xi(1) 
    - 
    \Hopf(\xi).
    \tag{HJ2}
    \label{eq:HJ2}
    \end{equation}
    For (HJ2), Theorem~4.7(3) in \cite{ChenXiaHJ2025} identifies the one-species
    Hopf formula with the unique Lipschitz viscosity solution, while
    Theorem~1.1 in \cite{ChenMourrat2025VectorNonconvex} identifies that solution with
    the limiting free energy. Their covariance
    class allows power series beginning at \(k=1\); see (1.24) in \cite{ChenMourrat2025VectorNonconvex}.
\end{enumerate}

\subsection*{Proof of \eqref{eq:balanced-superior}}

We can restrict the supremum in \eqref{eq:HJ1} to
the paths $\hat\sfp = (\lambda_1\sfp,\ldots,\lambda_D\sfp)$ for all $\sfp\in \PathSpace_\infty$.
Moreover, the homogeneity of each $\xi_k$ yields
\begin{equation}
    \xi(\hat\sfp)
    =
    \xi(\lambda_1\sfp,\ldots,\lambda_D\sfp)
    =
    \sum_{k\geq 1}
    \biggl(
    \sum_{\substack{\bm p\in\mathsf P_k}}
    \Delta_{\bmp}^2\prod_{d=1}^D \lambda_d^{p_d}
    \biggr) \sfp^k
    =
    \sum_{k\geq 1}\xi_k(\bmlambda)\sfp^k
    =
    \xionesp(\sfp),
\end{equation}
so we obtain the lower bound
\begin{equation}
\Hopf(\xi,\bmlambda)
\geq
    \sup_{\sfp\in \PathSpace_\infty}
    \inf_{\sfq\in \PathSpace_\infty^D}
    \left\{
    \sum_{d=1}^D \lambda_d \psi(\sfq_d)
    -
    \int_0^1 \sfp(s)
    \sum_{d=1}^D \lambda_d\sfq_d(s)\dd{s}
    +
    \frac{1}{2}
    \int_0^1 \xionesp(\sfp(s)) \dd{s}
    \right\}.
    \label{eq:Hopf.2}
\end{equation}
Now, note that we assume $\lambda_d>0$ for every $d\in\{1,\ldots,D\}$
and $\sum_{d=1}^D \lambda_d = 1$, which defines a probability measure on
$\{1,\ldots,D\}$.
Applying Jensen's inequality with this measure to $\psi$ yields
\begin{equation}
    \psi\Bigl(\sum_{d=1}^D \lambda_d \sfq_d\Bigr)
    \leq
    \sum_{d=1}^D \lambda_d \psi(\sfq_d).
    \label{eq:Jensen-psi}
\end{equation}
Combining \eqref{eq:Jensen-psi} with the inclusion
\begin{equation*}
    \overline{\PathSpace}_{D,\infty}
    \coloneqq
    \Bigl\{
        \bar\sfq
    \,\Big|\,
        \bar\sfq 
    = 
    \sum_{d=1}^D \lambda_d \sfq_d,
    \, \sfq\in \PathSpace_\infty^D
    \Bigr\}
    \subseteq 
    \PathSpace_\infty,
\end{equation*}
we obtain 
\begin{align}
    \inf_{\sfq\in \PathSpace_\infty^D}
    \left\{
    \sum_{d=1}^D \lambda_d \psi(\sfq_d)
    -
    \int_0^1 \sfp(s)
    \sum_{d=1}^D \lambda_d\sfq_d(s)\dd{s}
    \right\}
    &\geq 
    \inf_{\bar\sfq\in \overline{\PathSpace}_{D,\infty}}
    \left\{
    \psi(\bar\sfq)-\int_0^1 \sfp(s)\bar\sfq(s)\dd{s}
    \right\}
    \nonumber \\
    &\geq 
    \inf_{\sfq\in \PathSpace_\infty}
    \left\{
    \psi(\sfq)-\int_0^1 \sfp(s)\sfq(s)\dd{s}
    \right\}.
    \label{eq:Hopf-inf}
\end{align}
Combining \eqref{eq:HJ1}, \eqref{eq:Hopf.2} and \eqref{eq:Hopf-inf}
and noting that $\xi(\bmlambda)=\xionesp(1)$,
we obtain
\begin{equation*}
    \limsup_{N\rightarrow\infty} F_N(\xi,\bmlambda_N)
    \leq 
    \frac{1}{2}\xionesp(1) 
    - 
    \Hopf(\xionesp)
    =
    \lim_{N\rightarrow\infty} F_N(\xionesp),
\end{equation*}
where the equality above is provided by \eqref{eq:HJ2}.

\begin{remark}
    Note that the proof of \eqref{eq:balanced-superior} above does not use the balanced condition.
\end{remark}

\section{Time-changed Parisi PDE}
\label{sec:time-changed-parisi-pde}

As mentioned in the introduction, we want to perform a change of variables on a dense class with sufficient regularity
to simplify the analysis.
From now on, fix \(\alpha\in(0,1)\).
Introduce the following two classes of paths.
\begin{align}
\SmoothCore_\alpha
&=
\Bigl\{
\sfq\in C^{1+\alpha/2}([0,1])
\,\Big|\,
\sfq(0)=0,\ \inf_{[0,1]}\dot\sfq>0
\Bigr\}, \label{eq:def-smooth-core}
\\
\Opendenseset_\alpha
&=
\Bigl\{\gamma\in \Hoelder{\alpha/2}([0,1])
\,\Big|\,
\inf_{[0,1]}\gamma>0\Bigr\}.
\label{eq:def-open-dense-set}
\end{align}
Note that $\sfq\mapsto \dot\sfq$ maps \(\SmoothCore_\alpha\) to \(\Opendenseset_\alpha\).
Introduce the change of variables
\begin{equation}
    \label{eq:gamma-change-of-variables}
\Psi^{\dot\sfq}(s,x)=\Phi^{\sfq}(\sfq(s),x), 
\qquad (s,x)\in[0,1]\times\R.
\end{equation}
Given \(\gamma\in\Opendenseset_\alpha\), $\Psi^\gamma$ solves a time-changed Parisi PDE.
\begin{equation}\label{eq:gamma-pde}
\begin{cases}
-\partial_s\Psi^\gamma(s,x)
=
\dfrac12\gamma(s)\bigl(\partial_{xx}\Psi^\gamma(s,x)
+
s\,
(\partial_x\Psi^\gamma(s,x))^2\bigr),
& (s,x)\in [0,1)\times\R,
\\[0.5ex]
\Psi^\gamma(1,x)=\log\cosh x, & x\in\R.
\end{cases}
\end{equation}
Interestingly, \eqref{eq:gamma-pde}
already appeared in (16) of Parisi's original paper \cite{Parisi1980Sequence}.

Since the map $\sfq\mapsto \dot\sfq$ is linear,
to prove the concavity of $\sfq\mapsto\Phi^\sfq(0,0)$, it suffices to prove the concavity of the map $\gamma\mapsto\Psi^\gamma(0,0)$ defined 
on $\Opendenseset_\alpha$.
To achieve this, 
a natural strategy is to compute its second Fr\'echet derivative.
The standard theory of parabolic PDEs in H\"older spaces provides the needed regularity,
and we refer to Krylov \cite{Krylov1996} for a pedagogical treatment of this theory.

Based on the previous discussions, the following theorem implies
Theorem~\ref{thm:chen-initial-convexity}.
\begin{theorem}
\label{thm:concavity}
Denote by \(\ParaHoelder{\alpha}([0,1]\times\R)\) the parabolic H\"older space.
The following hold.
\begin{enumerate}
  \item \label{item:C2}
  The map \(\bm{\Psi}:\Opendenseset_\alpha \rightarrow \ParaHoelder{\alpha}([0,1]\times\mathbb R)\) defined by
  \[
  \bm{\Psi}(\gamma)(s,x)=\Psi^\gamma(s,x)-\log\cosh x,
  \qquad (s,x)\in [0,1]\times \R,
  \]
  is \(C^2\) in the Fr\'echet sense.
  \item \label{item:concavity}
  Fix $\gamma\in \Opendenseset_\alpha$ and $h\in \Hoelder{\alpha/2}([0,1])$. 
  Then, for all $s\in [0,1)$,
  \[
  D^2\Psi^\gamma[h,h](s,0)\leq 0.
  \]
  \item \label{item:off-diagonal}
  Fix $\gamma\in\Opendenseset_\alpha$ and fix $h_1,h_2\in \Hoelder{\alpha/2}([0,1])$ where $h_1\geq 0$ and $h_2\geq 0$.
  Then, for all $s\in [0,1)$,
  \[
  D^2\Psi^\gamma[h_1,h_2](s,0)\leq 0.
  \]
\end{enumerate}
\end{theorem}

\begin{remark}
Since the terminal condition of the time-changed Parisi PDE \eqref{eq:gamma-pde} is independent of $\gamma$, Theorem~\ref{thm:concavity}(1)
justifies the existence of the second Fr\'echet derivative of $\Psi^\gamma$.
Theorem~\ref{thm:concavity}(2) implies the concavity of $\gamma\mapsto\Psi^\gamma(0,0)$.
Theorem~\ref{thm:concavity}(3) provides off-diagonal information about the second derivative.
While it is not needed for the convexity of $\psi$, we record it for independent interest.
\end{remark}

\subsection*{Auffinger and Chen's raywise concavity}

Theorem~1 in \cite{AuffingerChen16} gives concavity only along the one-dimensional
ray generated by the fixed profile \(\gamma_0\).  
To explain this, we briefly recall their result.
Let $\xi(x) = \sum_{k=2}^\infty \Delta_k^2 x^k$ be the covariance function
of a one-species mixed model.
Let $\sfq$ be a quantile function on $[0,1]$ and $\sfq^{-1}$ be 
its generalized right inverse, or the corresponding CDF. 
Let $\Phi_{\beta}$ be the solution of the Parisi PDE 
\begin{equation}
  \label{eq:Parisi-pde}
\begin{cases}
-\partial_t\Phi_{\beta}(t,x)
=
\dfrac{\beta^2\xi''(t)}{2}\left(
\partial_{xx}\Phi_{\beta}(t,x)
+
\sfq^{-1}(t)(\partial_x\Phi_{\beta}(t,x))^2
\right),
& t\in [0,1),
\\[2ex]
\Phi_{\beta}(1,x)=\log\cosh x,
& x\in\R.
\end{cases}
\end{equation}
Theorem 1 in \cite{AuffingerChen16} proved that for fixed $\xi$ and $\sfq$, 
the map $\beta^2\mapsto\Phi_{\beta}(0,0)$ is concave.
Assuming for this discussion that \(\sfq\) is smooth and strictly increasing,
then $\Phi_\beta(\sfq(s),x)$ solves \eqref{eq:gamma-pde}
with $\gamma=\beta^2\gamma_0$ and $\gamma_0(s)=\dot{\sfq}(s)\xi''(\sfq(s))$.
In particular, 
\[
\Phi_\beta(0,0) = \Psi^{\beta^2\gamma_0}(0,0).
\]
Therefore, their result may be viewed
as the raywise concavity of
\[
    c\longmapsto \Psi^{c\gamma_0}(0,0).
\]




\subsection{Proof heuristics and outline of Theorem~\ref{thm:concavity}}
\label{sec:heuristics}

The proof is guided by the preceding interpretation of Auffinger--Chen's argument.
Rather than working in discretized time, we work directly with the
time-changed PDE \eqref{eq:gamma-pde} and its transition operators. First, 
Schauder's estimates in H\"older spaces and the Banach implicit function theorem 
give
the \(C^2\)-regularity of the solution map
\[
    \gamma\longmapsto \Psi^\gamma-\log\cosh ,
\]
which proves Theorem~\ref{thm:concavity}\textup{(1)}. Differentiating \eqref{eq:gamma-pde} with
respect to \(\gamma\) then gives linearized equations for the first and second
Fréchet derivatives.

The main step is to apply the Feynman--Kac formula to these variation equations and
rewrite the second derivative in a min-kernel form. 
This representation reduces the sign
of \(D^2\Psi^\gamma\) to the spatial symmetries preserved by the transition operator.
We leave the details to Section~\ref{sec:cone-preservation}.

\subsection*{Organization of the proof of Theorem~\ref{thm:concavity}}
In Section~\ref{sec:pde-notation} and Section~\ref{sec:solution-map-regularity}, we introduce the required machinery from the standard parabolic PDE theory and known properties of the solution
of the Parisi PDE. These allow us to prove Theorem~\ref{thm:concavity}(1) in
Section~\ref{subsec:section3-implicit-function}. Then, the PDEs for the 
Fr\'echet derivatives are provided in Section~\ref{subsec:variation-equations}. 
Using these PDEs and the Feynman--Kac formula, we derive the min-kernel representation in 
Section~\ref{sec:propagation}. Then, the details of the cone preservation argument
are in Section~\ref{sec:cone-preservation}, and this leads to the proof 
 of Theorem~\ref{thm:concavity}(2) and Theorem~\ref{thm:concavity}(3)
 in Section~\ref{sec:concavity-proof}.

\section{PDE preliminaries}
\label{sec:pde-notation}

This section collects the conventions used
in the rest of the paper. The notation in
Section~\ref{sec:hopf-formula-application} is independent of this section.

\subsection{Notation and conventions}
\label{sec:notation-conventions}
Unless otherwise stated, functions defined on only one factor of
\([0,1]\times\mathbb R\) are identified with their canonical lifts:
if \(a:[0,1]\to\mathbb R\), then \(a(s,x)=a(s)\), and if
\(f:\mathbb R\to\mathbb R\), then \(f(s,x)=f(x)\). We write
\(\|\cdot\|_\infty\) for the relevant supremum norm.

For
\(t\in[0,1]\), \(C([0,t];C_b^2(\mathbb R))\) denotes the space of
continuous maps 
\[r\mapsto u(r,\cdot)\in C_b^2(\mathbb R),\] 
equipped
with the supremum norm. We write \(C_b^{1,2}([0,t]\times\mathbb R)\)
for the space of functions that are \(C^1\) in time and \(C^2\) in
space, with all derivatives up to these orders bounded and continuous.
Classical solutions are understood in this sense.

For \(\gamma\in \Opendenseset_\alpha\), let \(\Psi^\gamma\) be the solution of the
Parisi PDE, and set
\[
    g(x)=\log\cosh x,\qquad
    \bm\Psi(\gamma)=\Psi^\gamma-g.
\]
For fixed \(\gamma\in \Opendenseset_\alpha\), define
\begin{align}
    L_s^\gamma u(s,x)
    &=
    \frac12\partial_{xx}u(s,x)
    +
    s\,\partial_x\Psi^\gamma(s,\cdot)\,\partial_x u(s,x),
    \label{eq:def-Lgamma}
    \\
    \paraop_s u(s,x)
    &=
    \partial_s u(s,x)
    +
    \gamma(s)L_s^\gamma u(s,x),
    \label{eq:def-Dgamma}
\end{align}
for all $(s,x)\in[0,1]\times\mathbb R$.

Finally, set
\[
    v^\gamma
    =
    \partial_{xx}\Psi^\gamma+s(\partial_x\Psi^\gamma)^2.
\]

\subsection{Parabolic PDEs in H\"older spaces}
\label{sec:Hoelder}

This section recalls the parabolic H\"older theory used throughout the paper, which includes standard facts of parabolic H\"older norms, 
the maximum principle, and Schauder estimates. Most of the contents are based on Krylov's lecture note \cite{Krylov1996}.

For a time interval \(I\subset[0,1]\), write
\[
    Q_I:=I\times\R .
\]
Below, when
\(I=[0,1]\) or \(Q_I=[0,1]\times\R\), we omit \(I\) and \(Q_I\) from the notation.

\subsubsection{Parabolic H\"older norms}
\label{sec:holder-estimates}

We write \(C_b^{\alpha/2}(I)\) for the Banach space of bounded
continuous functions \(\gamma:I\to\R\) such that
\[
  [\gamma]_{\alpha/2;I}
  :=
  \sup_{\substack{s,t\in I\\s\neq t}}
  \frac{|\gamma(t)-\gamma(s)|}{|t-s|^{\alpha/2}}
  <\infty ,
  \qquad
  \norm{\gamma}_{\alpha/2;I}
  :=
  \norm{\gamma}_{\infty;I}
  +
  [\gamma]_{\alpha/2;I}.
\]
On compact intervals this agrees with the usual H\"older space
\(C^{\alpha/2}(I)\).

For \(m\in\mathbb N_0\), we write \(C_b^m(\R)\) for the space of functions
whose derivatives \(\partial_x^j f\), \(0\le j\le m\), are bounded and
continuous.  We write \(C_b^{m+\alpha}(\R)\) for the subspace of
\(C_b^m(\R)\) such that \(\partial_x^m f\) is \(\alpha\)-H\"older, with norm
\[
  \norm{u}_{m+\alpha;\R}
  :=
  \sum_{j=0}^{m}
  \norm*{\partial_x^j u}_{\infty;\R}
  +
  [\partial_x^m u]_{\alpha;\R},
\]
where
\[
  [w]_{\alpha;\R}
  :=
  \sup_{\substack{x,y\in\R\\x\neq y}}
  \frac{|w(x)-w(y)|}{|x-y|^\alpha}.
\]

For \(u:Q_I\to\R\), define
\[
  [u]_{\alpha/2,t;Q_I}
  :=
  \sup_{x\in\R}
  \sup_{\substack{s,t\in I\\s\neq t}}
  \frac{|u(t,x)-u(s,x)|}{|t-s|^{\alpha/2}},
  \qquad
  [u]_{\alpha,x;Q_I}
  :=
  \sup_{t\in I}
  \sup_{\substack{x,y\in\R\\x\neq y}}
  \frac{|u(t,x)-u(t,y)|}{|x-y|^\alpha}.
\]
Set
\[
  [u]_{\alpha/2,\alpha;Q_I}
  :=
  [u]_{\alpha/2,t;Q_I}
  +
  [u]_{\alpha,x;Q_I},
  \qquad
  \norm{u}_{\alpha/2,\alpha;Q_I}
  :=
  \norm{u}_{\infty;Q_I}
  +
  [u]_{\alpha/2,\alpha;Q_I}.
\]
We write
\[
  C_b^{\alpha/2,\alpha}(Q_I)
  :=
  \bigl\{
      u\in C_b(Q_I):
      \norm{u}_{\alpha/2,\alpha;Q_I}<\infty
  \bigr\}.
\]

For \(m\in\mathbb N_0\), we write
\(C_b^{\alpha/2,m+\alpha}(Q_I)\) for the space of functions
\(u:Q_I\to\R\) such that, for every \(0\le j\le m\), the spatial derivative
\(\partial_x^j u\) exists and belongs to \(C_b^{\alpha/2,\alpha}(Q_I)\).
The norm is
\[
  \norm{u}_{\alpha/2,m+\alpha;Q_I}
  :=
  \sum_{j=0}^{m}
  \norm*{\partial_x^j u}_{\alpha/2,\alpha;Q_I}.
\]

Finally, we write \(C_b^{1+\alpha/2,2+\alpha}(Q_I)\) for the standard
second-order parabolic H\"older space of functions \(u\) whose derivatives
\(\partial_t^i\partial_x^j u\), \(0\le 2i+j\le2\), are bounded and continuous
on \(Q_I\), and whose top parabolic derivatives \(\partial_tu\) and
\(\partial_{xx}u\) belong to \(C_b^{\alpha/2,\alpha}(Q_I)\).  We equip it
with the norm
\[
  \norm{u}_{1+\alpha/2,2+\alpha;Q_I}
  :=
  \sum_{0\le 2i+j\le2}
  \norm*{\partial_t^i\partial_x^j u}_{\infty;Q_I}
  +
  [\partial_t u]_{\alpha/2,\alpha;Q_I}
  +
  [\partial_{xx}u]_{\alpha/2,\alpha;Q_I}.
\]

All spaces above are Banach spaces with their indicated norms.  

We will use without further comment the following standard facts.
\begin{enumerate}
  \item
  For all
  \(v,w\in\ParaHoelderzero{\alpha}([0,1]\times\R)\),
  \[
      \norm{vw}_{\alpha/2,\alpha}
      \le
      \norm{v}_{\alpha/2,\alpha}
      \norm{w}_{\alpha/2,\alpha}.
  \]

  \item
  There exists a constant \(C=C_\alpha>0\) such that
  for all \(v,w\in\ParaHoelder{\alpha}([0,1]\times\R)\),
  \[
      \norm{vw}_{1+\alpha/2,2+\alpha}
      \le
      C
      \norm{v}_{1+\alpha/2,2+\alpha}
      \norm{w}_{1+\alpha/2,2+\alpha}.
  \]

  \item
  There exists \(C=C_\alpha>0\) such that
  for all \(u\in\ParaHoelder{\alpha}([0,1]\times\R)\),
  \[
      \norm{u}_{\alpha/2,\alpha}
      \le
      C\norm{u}_{1+\alpha/2,2+\alpha}.
  \]
\end{enumerate}

We also record a lemma stating that the relevant derivatives define bounded operators between the corresponding parabolic H\"older spaces.

\begin{lemma}
\label{lem:X-to-Y-derivatives}
The maps
\[
  u\mapsto \partial_su,
  \qquad
  u\mapsto \partial_xu,
  \qquad
  u\mapsto \partial_{xx}u
\]
are bounded linear maps from $\ParaHoelder{\alpha}([0,1]\times\R)$ to $\ParaHoelderzero{\alpha}([0,1]\times\R)$.
There exists \(C=C_\alpha>0\) such that
for all $u\in\ParaHoelder{\alpha}([0,1]\times\R)$,
\[
  \|\partial_su\|_{\alpha/2,\alpha}
  +
  \|\partial_xu\|_{\alpha/2,\alpha}
  +
  \|\partial_{xx}u\|_{\alpha/2,\alpha}
  \le
  C\|u\|_{1+\alpha/2,2+\alpha}.
\]
\end{lemma}

\begin{proof}
Linearity is immediate. 
The estimates for \(\partial_su\) and \(\partial_{xx}u\) follow directly from
the definition of the $\ParaHoelder{\alpha}([0,1]\times\R)$-norm:
\[
  \|\partial_su\|_{\alpha/2,\alpha}
  =
  \|\partial_su\|_\infty
  +
  [\partial_su]_{\alpha/2,\alpha}
  \le
  \|u\|_{1+\alpha/2,2+\alpha},
\]
and
\[
  \|\partial_{xx}u\|_{\alpha/2,\alpha}
  =
  \|\partial_{xx}u\|_\infty
  +
  [\partial_{xx}u]_{\alpha/2,\alpha}
  \le
  \|u\|_{1+\alpha/2,2+\alpha}.
\]

It remains to derive the estimate for \(\partial_xu\). Applying the parabolic interpolation
inequality (cf. Theorem~8.8.1~in~\cite{Krylov1996})
yields that there exists a constant $C_\alpha>0$ such that
\[
  \|\partial_xu\|_{\alpha/2,\alpha}
  \le
  C_\alpha\|u\|_{1+\alpha/2,2+\alpha}.
\]
Combining the three estimates gives the claim.
\end{proof}

\subsubsection{Maximum principle}
\label{subsec:comparison-principle}

We will use the weak maximum principle in Krylov's lecture note only in the following form.  
Let \(\Domain\) be either \(\R\) or \((0,\infty)\).
Fix $T>0$. Define
\[
    Q_T^\Domain:=(0,T)\times\Domain.
\]
Define the terminal-time parabolic boundary by
\[
    \partial'_T Q_T^\Domain
    :=
    \bigl((0,T)\times\partial\Domain\bigr)
    \cup
    \bigl(\{T\}\times\overline\Domain\bigr).
\]

Let
\begin{equation}
    \LinearOp_t u
    =
    a_2(t,x)\partial_{xx}u
    +
    a_1(t,x)\partial_xu,
    \label{eq:LinearOp}
\end{equation}
where \(a_2,a_1\) are bounded and continuous on
\([0,T]\times\overline\Domain\), and \(a_2(t,x)\ge\kappa>0\).  Let \(a_0\) be
bounded and continuous on \([0,T]\times\overline\Domain\).

\begin{theorem}[Weak maximum principle]
\label{thm:comparison}
Suppose that \(u\) is bounded and continuous on
\([0,T]\times\overline\Domain\), is \(C^{1,2}\) in \(Q_T^\Domain\), and
satisfies
\begin{equation}
  \begin{cases}
    (-\partial_s-\LinearOp_s-a_0(s,x))u(s,x) \ge0,
    &
    (s,x)\in Q_T^\Domain,
    \\[0.5ex]
    u(s,x)\ge0,
    &
    (s,x)\in \partial'_TQ_T^\Domain.
  \end{cases}
\end{equation}
Then,
\[
    u\ge0
    \qquad\text{on }[0,T]\times\overline\Domain .
\]
\end{theorem}
\begin{proof}
Choose \(\lambda>\norm{a_0}_\infty\), and set
\[
    \tilde u(t,x):=e^{-\lambda t}u(T-t,x).
\]
Define the time-reversed coefficients
\[
    \tilde a_2(t,x):=a_2(T-t,x),
    \qquad
    \tilde a_1(t,x):=a_1(T-t,x),
    \qquad
    \tilde a_0(t,x):=a_0(T-t,x),
\]
and
\[
    \tilde{\LinearOp}_t v
    :=
    \tilde a_2(t,x)\partial_{xx}v
    +
    \tilde a_1(t,x)\partial_xv .
\]
Then \(\tilde u\ge0\) on the forward parabolic boundary
\[
    \partial'_0Q_T^\Domain
    :=
    \bigl((0,T)\times\partial\Domain\bigr)
    \cup
    \bigl(\{0\}\times\overline\Domain\bigr),
\]
and
\[
    \bigl(
    \partial_t-\tilde{\LinearOp}_t-(\tilde a_0-\lambda)
    \bigr)\tilde u
    =
    e^{-\lambda t}
    \bigl[
    (-\partial_s-\LinearOp_s-a_0)u
    \bigr](T-t,x)
    \ge0 .
\]
Since \(\tilde a_0-\lambda\le0\), applying Theorem~8.1.4~in~\cite{Krylov1996} yields
\(-\tilde u\le0\) on
\([0,T]\times\overline\Domain\). 
Reversing the sign and undoing the damping and time reversal then give
\(u\ge0\) on
\([0,T]\times\overline\Domain\).
\end{proof}

\begin{remark}
In the applications below, \(\LinearOp_s\) will usually be
\[
    \LinearOp_s=\gamma(s)L_s^\gamma,
\]
that is,
\[
    a_2(s,x)=\frac12\gamma(s),
    \qquad
    a_1(s,x)=\gamma(s)s\,\partial_x\Psi^\gamma(s,x).
\]
The zeroth-order coefficient \(a_0\) is usually zero, except after
differentiating in \(x\), where a bounded coefficient appears.
\end{remark}

\subsubsection{Schauder estimates}
\label{subsec:schauder-estimates}

The following terminal-time estimate is the form of Krylov's whole-space
Schauder theorem used below.

\begin{theorem}[Terminal-time Schauder estimate]
\label{thm:terminal-time-schauder}
Let $\LinearOp_s$ be as in \eqref{eq:LinearOp} with 
\[
    a_2,a_1\in \ParaHoelderzero{\alpha}([0,1]\times\R),
    \qquad
    \inf_{[0,1]\times\R} a_2\ge\kappa>0.
\]
Then, for every
\[
    b\in \ParaHoelderzero{\alpha}([0,1]\times\R),
    \qquad
    \varphi\in C_b^{2+\alpha}(\R),
\]
the terminal-value problem
\[
\begin{cases}
-\partial_su(s,x)=\LinearOp_su(s,x)+b(s,x),
& (s,x)\in[0,1)\times\R,\\[0.5ex]
u(1,x)=\varphi(x),
& x\in\R,
\end{cases}
\]
has a unique solution
\[
    u\in \ParaHoelder{\alpha}([0,1]\times\R).
\]
Moreover,
there exists a constant \(C\) depending only on \(\alpha,\kappa\), and the
\(C_b^{\alpha/2,\alpha}\)-bounds of \(a_2,a_1\)
such that
\[
    \norm{u}_{1+\alpha/2,2+\alpha}
    \le
    C
    \left(
        \norm{b}_{\alpha/2,\alpha}
        +
        \norm{\varphi}_{C_b^{2+\alpha}(\R)}
    \right),
\]
\end{theorem}

\begin{proof}
Set \(w(t,x)=u(1-t,x)\).  Then \(w\) solves the forward Cauchy problem
\[
    \partial_t w
    =
    a_2(1-t,x)\partial_{xx}w
    +
    a_1(1-t,x)\partial_xw
    +
    b(1-t,x),
    \qquad
    w(0,\cdot)=\varphi .
\]
Let \(\tilde w(t,x)=e^{-t}w(t,x)\).  Then
\[
    \partial_t \tilde w
    =
    a_2(1-t,x)\partial_{xx} \tilde w
    +
    a_1(1-t,x)\partial_x \tilde w
    -
    \tilde w
    +
    e^{-t}b(1-t,x),
    \qquad
    \tilde w(0,\cdot)=\varphi .
\]
The zeroth-order coefficient is now \(-1\), so Theorem~9.2.3~in~\cite{Krylov1996} applies.  Undoing the damping and
time reversal gives the stated solution and estimate.
\end{proof}

\subsection{Preliminaries for the Parisi PDE}
\label{sec:preliminaries}

This section collects the basic facts about the Parisi PDE used
throughout the paper.  Most of these facts are standard, and we recall them here
only to fix conventions.

The following lemma connects the time-changed Parisi PDE \eqref{eq:gamma-pde} 
to the conventional Parisi PDE.

\begin{lemma}
\label{lem:gamma-conventional-change-of-variable}
Fix \(\gamma\in\Opendenseset_\alpha\), and set
\[
    \sfq(s):=\int_0^s\gamma(r)\,\dd r,
    \qquad
    q_\gamma:=\sfq(1),
    \qquad
    \sfp(s):=\frac{\sfq(s)}{q_\gamma},
    \qquad
    \zeta:=\sfp^{-1}.
\]
Let \(\Phi_\gamma\) solves the following conventional Parisi PDE
\[
\begin{cases}
-\partial_t\Phi_\gamma(t,x)
=
\dfrac{q_\gamma}{2}
\left(
\partial_{xx}\Phi_\gamma(t,x)
+
\zeta(t)\bigl(\partial_x\Phi_\gamma(t,x)\bigr)^2
\right),
& (t,x)\in[0,1)\times\mathbb R,
\\[0.5ex]
\Phi_\gamma(1,x)=\log\cosh x,
& x\in\mathbb R.
\end{cases}
\]
Then,
\[
    \Psi^\gamma(s,x)=\Phi_\gamma(\sfp(s),x),
    \qquad (s,x)\in[0,1]\times\mathbb R .
\]
\end{lemma}

\begin{proof}
Since \(\gamma\in\Opendenseset_\alpha\), we have \(q_\gamma>0\) and
\[
    \dot\sfp(s)=\frac{\gamma(s)}{q_\gamma}>0.
\]
Hence \(\sfp\in C^{1+\alpha/2}([0,1])\). Moreover,
\(\sfp(0)=0\), \(\sfp(1)=1\), and
\[
    \dot\sfp(s)\ge
    \frac{\inf_{[0,1]}\gamma}{q_\gamma}>0.
\]
Thus \(\sfp\) is strictly increasing and maps \([0,1]\) bijectively onto
itself. The inverse function theorem gives \(\zeta\in C^1([0,1])\), with
\[
    \dot\zeta(t)
    =
    \frac{q_\gamma}{\gamma(\zeta(t))}.
\]
Since \(1/\gamma\in C^{\alpha/2}([0,1])\) and \(\zeta\) is Lipschitz,
\(\dot\zeta\in C^{\alpha/2}([0,1])\). Therefore
\(\zeta\in C^{1+\alpha/2}([0,1])\), and \(\sfp\) is a
\(C^{1+\alpha/2}\)-diffeomorphism of \([0,1]\) onto itself.

Set
\[
    \widetilde\Psi(s,x):=\Phi_\gamma(\sfp(s),x).
\]
Since \(\sfp(1)=1\), the terminal condition gives
\[
    \widetilde\Psi(1,x)=\Phi_\gamma(1,x)=\log\cosh x.
\]
For \(s\in[0,1)\), the chain rule and the conventional Parisi PDE give
\[
\begin{aligned}
-\partial_s\widetilde\Psi(s,x)
&=
-\dot\sfp(s)\,\partial_t\Phi_\gamma(\sfp(s),x)
\\
&=
\frac{\gamma(s)}{q_\gamma}\cdot\frac{q_\gamma}{2}
\left(
\partial_{xx}\Phi_\gamma(\sfp(s),x)
+
\zeta(\sfp(s))
\bigl(\partial_x\Phi_\gamma(\sfp(s),x)\bigr)^2
\right)
\\
&=
\frac{\gamma(s)}2
\left(
\partial_{xx}\widetilde\Psi(s,x)
+
s\bigl(\partial_x\widetilde\Psi(s,x)\bigr)^2
\right).
\end{aligned}
\]
Thus \(\widetilde\Psi\) solves \eqref{eq:gamma-pde} with terminal condition
\(\log\cosh x\). By uniqueness of the classical solution to \eqref{eq:gamma-pde},
\(\widetilde\Psi=\Psi^\gamma\), proving the claim.
\end{proof}

\subsubsection{Regularity of the Parisi solution}
\label{sec:regularity}

The next lemma collects  the standard spatial regularity, symmetry, and derivative
bounds for the Parisi solution that will be used throughout the sequel.
\begin{lemma}
\label{lem:Parisi-standard}
Let \(\gamma\in\Opendenseset_\alpha\). Then the following hold.

\begin{enumerate}
\item For every \(j\ge1\),
\[
\partial_x^j\Psi^\gamma\in C_b([0,1]\times\mathbb R),
\]
and, for every \(j\ge0\),
\[
\partial_s\partial_x^j\Psi^\gamma
\in L^\infty([0,1]\times\mathbb R).
\]
In particular, for every $j\geq 1$,
\[
\partial_x^j\Psi^\gamma\in
C_b^{\alpha/2,\alpha}([0,1]\times\mathbb R).
\]

\item For every \(s\in[0,1]\), the function
\(x\mapsto\Psi^\gamma(s,x)\) is even,
\(x\mapsto\partial_x\Psi^\gamma(s,x)\) is odd, and
\(x\mapsto\partial_{xx}\Psi^\gamma(s,x)\) is even.

\item For all \((s,x)\in[0,1]\times\mathbb R\),
\[
|\partial_x\Psi^\gamma(s,x)|\le1,
\qquad
0\le \partial_{xx}\Psi^\gamma(s,x)\le1,
\qquad
|\partial_{xxx}\Psi^\gamma(s,x)|\le4.
\]
In particular, \(x\mapsto\Psi^\gamma(s,x)\) is convex.
\end{enumerate}
\end{lemma}

\begin{proof}
With the notation of Lemma~\ref{lem:gamma-conventional-change-of-variable},
\[
    \Psi^\gamma(s,x)
    =
    \Phi_\gamma(\sfp(s),x),
    \qquad
    \xi_\gamma(t)=\frac{q_\gamma}{2}t^2 .
\]
Therefore the regularity property for the conventional Parisi PDE, for instance
Theorem~4~in~\cite{JagannathTobasco2016}, gives
\[
    \partial_x^j\Psi^\gamma\in C_b([0,1]\times\mathbb R),
    \qquad j\ge1.
\]
Using the equation
\[
-\partial_s\Psi^\gamma
=
\frac12\gamma(s)
\left(
\partial_{xx}\Psi^\gamma
+
s(\partial_x\Psi^\gamma)^2
\right),
\]
and differentiating \(j\) times in \(x\), we get
\[
\partial_s\partial_x^j\Psi^\gamma
=
-\frac12\gamma(s)
\partial_x^j
\left(
\partial_{xx}\Psi^\gamma
+
s(\partial_x\Psi^\gamma)^2
\right).
\]
The right-hand side is bounded because \(\gamma\) is bounded and all positive
spatial derivatives of \(\Psi^\gamma\) are bounded.  Hence
\[
    \partial_s\partial_x^j\Psi^\gamma
    \in L^\infty([0,1]\times\mathbb R),
    \qquad j\ge0.
\]
In particular, for all $j\geq 1$, \(\partial_x^j\Psi^\gamma\) is Lipschitz in both \(s\) and \(x\);
therefore
\[
    \partial_x^j\Psi^\gamma
    \in C_b^{\alpha/2,\alpha}([0,1]\times\mathbb R),
    \qquad j\ge1.
\]

The terminal datum \(x\mapsto\log\cosh x\) is even, and the equation is
invariant under \(x\mapsto -x\).  By uniqueness,
\[
    \Psi^\gamma(s,x)=\Psi^\gamma(s,-x).
\]
Differentiating in \(x\) gives the parity assertions.

Finally, by the same change of variables and the standard Parisi derivative
bounds, see Proposition~2~in~\cite{AuffingerChen2015} and
(14.272)~in~\cite{Talagrand2011},
\[
|\partial_x\Psi^\gamma|\le1,
\qquad
0\le\partial_{xx}\Psi^\gamma\le1,
\qquad
|\partial_{xxx}\Psi^\gamma|\le4.
\]
The convexity of \(x\mapsto\Psi^\gamma(s,x)\) follows from
\(\partial_{xx}\Psi^\gamma\ge0\).
\end{proof}

\section{Regularity of the Parisi solution map}
\label{sec:solution-map-regularity}

The purpose of this section is to prove Theorem~\ref{thm:concavity}(1),
which goes through Schauder estimates plus the Banach implicit function theorem.

Recall that 
\begin{equation*}
g(x)=\log\cosh x, \qquad \bm\Psi(\gamma)=\Psi^\gamma-g
\end{equation*}
for $\gamma\in\Opendenseset_\alpha$.

For $t\in [0,1]$, we introduce the Banach subspace of $\ParaHoelder{\alpha}([0,1]\times\R)$
defined by
\[
  \mathcal X_t
  :=
  \Bigl\{
    u\in \ParaHoelder{\alpha}([0,1]\times\R):
    u(t,\cdot)=0
  \Bigr\}
  \subseteq
  \ParaHoelder{\alpha}([0,1]\times\R).
\]

Before introducing the nonlinear residual map, we verify that the normalized
Parisi solution $\bm\Psi(\gamma)$ actually lies in $\mathcal X_1$.  

\begin{lemma}
\label{lem:parisi-schauder-class}
Let \(\gamma\in\Opendenseset_\alpha\). Then
\[
\bm\Psi(\gamma)=\Psi^\gamma-\log\cosh
\in
\mathcal X_1.
\]
\end{lemma}

\begin{proof}
Define
\[
u(s,x):=\bm\Psi(\gamma)(s,x)=\Psi^\gamma(s,x)-g(x).
\]
Since \(\Psi^\gamma(1,\cdot)=g\), we have \(u(1,\cdot)=0\).  It remains to
show that \(u\in\ParaHoelder{\alpha}([0,1]\times\R)\).

By Lemma~\ref{lem:Parisi-standard}(1),
\[
\partial_x\Psi^\gamma\in
C_b^{\alpha/2,\alpha}([0,1]\times\R).
\]
The function \(u\) satisfies
\begin{equation}
\label{eq:normalized-parisi-linear-equation}
-\partial_su
-
\frac12\gamma(s)\partial_{xx}u
=
b(s,x),
\qquad
u(1,\cdot)=0,
\end{equation}
where
\[
b(s,x)
\coloneqq
\frac12\gamma(s)
\left(
g''(x)+s(\partial_x\Psi^\gamma(s,x))^2
\right).
\]
Note that \(b\in \ParaHoelderzero{\alpha}([0,1]\times\R)\).
Indeed, \(g''=\sech^2\in C_b^\infty(\R)\), \(\gamma\in C^{\alpha/2}([0,1])\),
\(\partial_x\Psi^\gamma\in C_b^{\alpha/2,\alpha}\), and
\(C_b^{\alpha/2,\alpha}\) is a Banach algebra.  Applying
Theorem~\ref{thm:terminal-time-schauder} gives
\(u\in\ParaHoelder{\alpha}([0,1]\times\R)\), and hence \(u\in\mathcal X_1\).
\end{proof}

After subtracting the terminal datum \(g=\log\cosh\), the Parisi PDE can be
written as the zero set of a nonlinear residual map on \(\mathcal{X}_1\).
This allows us
to apply the Banach implicit function theorem.

\begin{proposition}
\label{prop:solution-map-C2}
Define the map $\Residual:\Opendenseset_\alpha\times \mathcal{X}_1\rightarrow \ParaHoelderzero{\alpha}([0,1]\times\R)$
by
\begin{equation}
\label{eq:nonlinear-residual-map}
  \Residual(\gamma,u)
  \coloneqq
  -\partial_su
  -
  \frac12\gamma(s)\mathfrak A(u),
  \qquad 
  \mathfrak A(u)
  \coloneqq
  \partial_{xx}u+g''+s(\partial_xu+g')^2.
\end{equation}
Then, $\Residual$ satisfies the following properties.
\begin{enumerate}
\item The map $\Residual$ is well-defined and 
$\Residual\in  C^2
  \bigl(
  \Opendenseset_\alpha\times\mathcal X_1
  ;\, \ParaHoelderzero{\alpha}([0,1]\times\R)
  \bigr)$.
\item For every \(\gamma\in\Opendenseset_\alpha\),
\[
  \Residual(\gamma,\Psi^\gamma-\log\cosh)=0.
\]

\item Let $\mathfrak L_s^u\coloneqq \frac{1}{2}\partial_{xx}+s(\partial_x u + g')\partial_x$.
Then,
\begin{align}
  D_\gamma\Residual(\gamma,u)[h]
  &=
  -\frac12 h(s)
  \mathfrak A(u), 
  \label{eq:DN-gamma-formula} \\
  D_u\Residual(\gamma,u)[v]
  &=
  -\partial_sv
  -
  \gamma(s)
  \mathfrak L_s^u v,
  \label{eq:DN-u-formula} \\
  D^2\Residual(\gamma,u)[(h_1,v_1),(h_2,v_2)]
  &=
-
h_1(s)
\mathfrak L_s^u v_2
-
h_2(s)
\mathfrak L_s^u v_1
-
\gamma(s)s\,\partial_xv_1\,\partial_xv_2.
\label{eq:DN2-formula}
\end{align}
\end{enumerate}
\end{proposition}

\begin{remark}
    Lemma~\ref{lem:parisi-schauder-class} ensures
that the expression \(\Residual(\gamma,\Psi^\gamma-\log\cosh)\) 
in Proposition~\ref{prop:solution-map-C2}(2) is well-defined.
\end{remark}

\begin{proof}
We assert the proposition as follows.
\paragraph{\emph{Proof of (1).}}
By Lemma~\ref{lem:X-to-Y-derivatives}, the maps
\[
u\mapsto \partial_su, 
\qquad
u\mapsto \partial_xu,
\qquad
u\mapsto \partial_{xx}u
\] 
are bounded linear maps from \(\mathcal X_1\) to
\(\ParaHoelderzero{\alpha}([0,1]\times\R)\), 
so
\[
u\mapsto\partial_xu+g',
\qquad
u\mapsto\partial_{xx}u+g''
\] 
are affine continuous
from \(\mathcal X_1\) to
\(\ParaHoelderzero{\alpha}([0,1]\times\R)\).  
Therefore,
\[
    \Residual(\gamma,u)
    =
    -\partial_s u
    -
    \frac{1}{2}\gamma(s)\mathfrak A(u)
    =
    -\partial_s u
    -
    \frac{1}{2}\gamma(s)
    \bigl(
    \partial_{xx}u+g''+s(\partial_xu+g')^2
    \bigr)
\]
belongs to \(C^2(\Opendenseset_\alpha\times \mathcal X_1;\, \ParaHoelderzero{\alpha}([0,1]\times\R))\),
as squaring is a \(C^2\) operation
and multiplication by the lifted function \(s\) is bounded linear.

\paragraph{\emph{Proof of (2).}}
Let
\[
  u=\Psi^\gamma-g.
\]
By Lemma~\ref{lem:parisi-schauder-class}, \(u\in\mathcal X_1\). Moreover,
\[
  \partial_su=\partial_s\Psi^\gamma,
  \qquad
  \partial_xu+g'=\partial_x\Psi^\gamma,
  \qquad
  \partial_{xx}u+g''=\partial_{xx}\Psi^\gamma.
\]
Therefore
\[
  \mathfrak A(u)
  =
  \partial_{xx}\Psi^\gamma
  +
  s(\partial_x\Psi^\gamma)^2.
\]
Consequently, by \eqref{eq:gamma-pde},
\begin{equation*}
  \Residual(\gamma,\Psi^\gamma-g)
  =
  -\partial_s\Psi^\gamma
  -
  \frac12\gamma(s)
  \bigl(
    \partial_{xx}\Psi^\gamma
    +
    s(\partial_x\Psi^\gamma)^2
  \bigr)
  =0.
\end{equation*}

\paragraph{\emph{Proof of (3).}}
Let
\[
  h,h_1,h_2\in\Hoelder{\alpha/2}([0,1]),
  \qquad
  v,v_1,v_2\in\mathcal X_1.
\]
Since \(\Residual\) is affine in \(\gamma\),
\[
  D_\gamma\Residual(\gamma,u)[h]
  =
  -\frac12h(s)\mathfrak A(u),
\]
which is \eqref{eq:DN-gamma-formula}.

Since
  $D\mathfrak A(u)[v]=2\mathfrak L_s^u v$,
one has
\begin{equation*}
  D_u\Residual(\gamma,u)[v]
  =
  -\partial_sv
  -
  \frac12\gamma(s)D\mathfrak A(u)[v]
  =
  -\partial_sv-\gamma(s)\mathfrak L_s^u v,
\end{equation*}
which is \eqref{eq:DN-u-formula}.

Finally, using
\[
  D^2\mathfrak A(u)[v_1,v_2]
  =
  2s\,\partial_xv_1\,\partial_xv_2,
\]
and differentiating the preceding identities, we obtain
\begin{equation*}
D^2\Residual(\gamma,u)[(h_1,v_1),(h_2,v_2)]
=
-h_1(s)\mathfrak L_s^u v_2
-h_2(s)\mathfrak L_s^u v_1
-\gamma(s)s\,\partial_xv_1\,\partial_xv_2,
\end{equation*}
which is \eqref{eq:DN2-formula}.
\end{proof}

We next verify that the isomorphism condition
required to apply the implicit function theorem holds.

\begin{proposition}
\label{prop:linearized-isomorphism}
For every \(\gamma\in\Opendenseset_\alpha\), the linear map
\[
  D_u \Residual(\gamma,\Psi^\gamma-\log\cosh):\mathcal X_1\rightarrow\ParaHoelderzero{\alpha}([0,1]\times\R)
\]
is a Banach space isomorphism. Moreover,
\begin{equation}
\label{eq:DuN-linearized-operator}
  D_u\Residual(\gamma,\Psi^\gamma-\log\cosh)[v]
  =
  -\paraop_s v.
\end{equation}
\end{proposition}

\begin{proof}
Set
\[
    u_\gamma:=\Psi^\gamma-g .
\]
By \eqref{eq:DN-u-formula} in Proposition~\ref{prop:solution-map-C2}, for
\(v\in\mathcal X_1\),
\[
  D_u\Residual(\gamma,u_\gamma)[v]
  =
  -\partial_sv-\gamma(s)\mathfrak L_s^{u_\gamma}v.
\]
Since
\[
\mathfrak L_s^{u_\gamma}
=
\frac12\partial_{xx}
+
s(\partial_xu_\gamma+g')\partial_x
=
\frac12\partial_{xx}
+
s\partial_x\Psi^\gamma\,\partial_x
=
L_s^\gamma,
\]
we obtain
\[
  D_u\Residual(\gamma,\Psi^\gamma-\log\cosh)[v]
  =
  -\partial_sv-\gamma(s)L_s^\gamma v
  =
  -\paraop_s v.
\]
This proves \eqref{eq:DuN-linearized-operator}.

It remains to show that
$
    -\paraop_s:\mathcal X_1
    \rightarrow
    \ParaHoelderzero{\alpha}([0,1]\times\R)
$
is a Banach-space isomorphism.  

Let
\[
    f\in\ParaHoelderzero{\alpha}([0,1]\times\R).
\]
The equation
\begin{equation}
    -\paraop_s v=f,
    \qquad
    v(1,\cdot)=0,
    \label{eq:Dvf}
\end{equation}
is equivalent to
\[
    -\partial_sv
    =
    \frac12\gamma(s)\partial_{xx}v
    +
    \gamma(s)s\,\partial_x\Psi^\gamma(s,x)\partial_xv
    +
    f.
\]
This is exactly the form covered by
Theorem~\ref{thm:terminal-time-schauder}, with
\[
    a_2(s,x)=\frac12\gamma(s),
    \qquad
    a_1(s,x)=\gamma(s)s\,\partial_x\Psi^\gamma(s,x),
    \qquad
    b=f,
    \qquad
    \varphi=0.
\]
The coefficient \(a_2\) is uniformly elliptic because
\(\inf_{[0,1]}\gamma>0\).  Moreover,
\[
    a_2,
    a_1\in\ParaHoelderzero{\alpha}([0,1]\times\R),
\]
where the second inclusion follows from Lemma~\ref{lem:Parisi-standard}, the
lift convention, and the Banach algebra property of
\(\ParaHoelderzero{\alpha}([0,1]\times\R)\).

Therefore Theorem~\ref{thm:terminal-time-schauder} gives a unique
\[
    v\in\ParaHoelder{\alpha}([0,1]\times\R)
\]
solving the terminal problem \eqref{eq:Dvf}, and
\[
    \norm{v}_{1+\alpha/2,2+\alpha}
    \le
    C_\gamma\norm{f}_{\alpha/2,\alpha}.
\]
Since \(v(1,\cdot)=0\), \(v\in \mathcal X_1\).  
Hence,
\(-\paraop_s\) is surjective and has bounded inverse.  

Injectivity follows
from the uniqueness part of Theorem~\ref{thm:terminal-time-schauder}.  Thus
\[
    -\paraop_s:\mathcal X_1
    \rightarrow
    \ParaHoelderzero{\alpha}([0,1]\times\R)
\]
is a Banach-space isomorphism.  By
\eqref{eq:DuN-linearized-operator}, the same is true for
\(D_u\Residual(\gamma,\Psi^\gamma-\log\cosh)\).
\end{proof}

\subsection{Proof of Theorem~\ref{thm:concavity}(1)}
\label{subsec:section3-implicit-function}

  The requirements of applying the Banach implicit function theorem 
  (cf. Theorem~5.9~in~\cite{Lang1999}) are fulfilled by (1) and (2) 
  from Proposition~\ref{prop:solution-map-C2} and Proposition~\ref{prop:linearized-isomorphism}.
  This yields that for every \(\gamma_0\in\Opendenseset_\alpha\), there exist a neighbourhood
\(\mathcal V\subset\Opendenseset_\alpha\) of \(\gamma_0\) and a unique \(C^2\) map
\[
\mathcal V\ni\gamma\longmapsto u(\gamma)\in\mathcal X_1
\]
such that
\[
\Residual(\gamma,u(\gamma))=0,
\qquad
u(\gamma_0)=\Psi^{\gamma_0}-\log\cosh=\bm{\Psi}(\gamma_0).
\]
Therefore, $\bm\Psi$ is $C^2$ in the Fr\'echet sense, as desired. 

\subsection{PDEs for the Fr\'echet derivatives}
\label{subsec:variation-equations}

Since \(g\) is independent of \(\gamma\), we define, for
\(h,h_1,h_2\in\Hoelder{\alpha/2}([0,1])\),
\[
  V_h^\gamma
  =
  D\Psi^\gamma[h]
  :=
  D\bm\Psi(\gamma)[h],
\]
and
\[
  W_{h_1,h_2}^\gamma
  :=
  D^2\Psi^\gamma[h_1,h_2]
  :=
  D^2\bm\Psi(\gamma)[h_1,h_2].
\]

\begin{corollary}[First and second variation equations]
\label{cor:variation-equations}
Let \(\gamma\in\Opendenseset_\alpha\) and
\(h,h_1,h_2\in\Hoelder{\alpha/2}([0,1])\). Then
\(V_h^\gamma\in\mathcal X_1\) is the unique solution of
\begin{equation}
\label{eq:linearized-parisi}
\begin{cases}
-\partial_s V_h^\gamma
=
\gamma(s)L_s^\gamma V_h^\gamma
+
\dfrac12h(s)v^\gamma,
& (s,x)\in[0,1)\times\R,
\\[0.6ex]
V_h^\gamma(1,x)=0,
& x\in\R,
\end{cases}
\end{equation}
where \(v^\gamma\) is the quantity fixed in Section~\ref{sec:preliminaries}.
Moreover, \(W_{h_1,h_2}^\gamma\in\mathcal X_1\) is the unique solution of
\begin{equation}
\label{eq:second-linearized-parisi}
\begin{cases}
\begin{aligned}
-\partial_s W_{h_1,h_2}^\gamma
={}&
\gamma(s)L_s^\gamma W_{h_1,h_2}^\gamma
+
h_1(s)L_s^\gamma V_{h_2}^\gamma
+
h_2(s)L_s^\gamma V_{h_1}^\gamma
\\
&+
\gamma(s)s\,\partial_xV_{h_1}^\gamma\,\partial_xV_{h_2}^\gamma,
\end{aligned}
& (s,x)\in[0,1)\times\R,
\\[0.6ex]
W_{h_1,h_2}^\gamma(1,x)=0,
& x\in\R.
\end{cases}
\end{equation}
The map \(h\mapsto V_h^\gamma\) is linear, and
\((h_1,h_2)\mapsto W_{h_1,h_2}^\gamma\) is symmetric bilinear.
\end{corollary}

\begin{proof}
Set \(u:=\bm\Psi(\gamma)=\Psi^\gamma-g\). Since
\[
  \Residual(\gamma,\bm\Psi(\gamma))=0,
  \qquad \gamma\in\Opendenseset_\alpha,
\]
and since both \(\Residual\) and \(\bm\Psi\) are \(C^2\), we may differentiate
this identity in \(\gamma\).

Differentiating once in the direction \(h\) gives
\[
  0=
  D_\gamma\Residual(\gamma,u)[h]
  +
  D_u\Residual(\gamma,u)[V_h^\gamma].
\]
Using Proposition~\ref{prop:solution-map-C2}(3), together with
\[
  \mathfrak A(u)=v^\gamma,
  \qquad
  \mathfrak L_s^u=L_s^\gamma,
\]
gives \eqref{eq:linearized-parisi}. Since
\(\bm\Psi(\gamma)(1,\cdot)=0\) for all \(\gamma\in\Opendenseset_\alpha\), one has
\(V_h^\gamma(1,\cdot)=0\). Uniqueness follows from
Proposition~\ref{prop:linearized-isomorphism}.

Differentiating the identity
\[
  \Residual(\gamma,\bm\Psi(\gamma))=0
\]
twice in directions \(h_1,h_2\) gives
\[
  0=
  D_u\Residual(\gamma,u)[W_{h_1,h_2}^\gamma]
  +
  D^2\Residual(\gamma,u)
  [
    (h_1,V_{h_1}^\gamma),
    (h_2,V_{h_2}^\gamma)
  ].
\]
Using again Proposition~\ref{prop:solution-map-C2}(3) and
\(\mathfrak L_s^u=L_s^\gamma\) gives
\eqref{eq:second-linearized-parisi}. The terminal condition follows from
\(\bm\Psi(\gamma)(1,\cdot)=0\), and uniqueness again follows from
Proposition~\ref{prop:linearized-isomorphism}. Linearity and symmetric
bilinearity follow from the corresponding properties of
\(D\bm\Psi(\gamma)\) and \(D^2\bm\Psi(\gamma)\).
\end{proof}

\section{Min-kernel representation of \texorpdfstring{$D^2\Psi^\gamma$}{Hessian}}
\label{sec:propagation}

The goal of this section is to rewrite \(D^2\Psi^\gamma\) in a form suitable to study its sign.
In view of Corollary~\ref{cor:variation-equations} and the standard Feynman--Kac
formula, the right framework is the study of the transition operator for diffusion processes.
While this topic is standard, 
we recall the relevant facts in Section~\ref{sec:tools-kernel}. 
The derivation of the representation is provided in
Section~\ref{sec:min-kernel-proof}. Finally, we provide the implications of that representation in 
Section~\ref{sec:cor-min-kernel}.

\subsection{Tools for the kernel representation}
\label{sec:tools-kernel}

This section collects the tools for the computation to derive the min-kernel representation
of \(D^2\Psi^\gamma\) in Section~\ref{sec:min-kernel-proof}.

\subsubsection{Standard properties of the transition operators}
\label{subsec:transition}

We first record the standard transition-operator facts used in the derivation of the
min-kernel representation. Recall from \eqref{eq:def-Lgamma} and \eqref{eq:def-Dgamma} that
for all $(s,x)\in [0,1]\times \R$,
\begin{align}
L_s^\gamma u(s,x)
&=
\frac12 \partial_{xx}u(s,x)
+
s\,\partial_x\Psi^\gamma(s,x)\partial_x u(s,x), \\
\paraop_s u(s,x)
    &=
    \partial_s u(s,x)
    +
    \gamma(s)L_s^\gamma u(s,x).
\end{align}
Fix \(0\le s< \sigma\le1\). 
The family of operators $\bm{\mathcal{L}}_{s,\sigma}^\gamma = (\gamma(\rho)L_\rho^\gamma)_{\rho\in [s,\sigma]}$
has the associated diffusion process \(\bm{X}^{\gamma,s,\sigma,x}=(X_\rho^{\gamma,s,\sigma,x})_{\rho\in [s,\sigma]}\) defined by the following SDE
\begin{equation}
\label{eq:X-sde}
\begin{cases}
\dd X_\rho^{\gamma,s,\sigma,x}
=
\gamma(\rho)\rho\,
\partial_x\Psi^\gamma(\rho,X_\rho^{\gamma,s,\sigma,x})\,\dd \rho
+
\sqrt{\gamma(\rho)}\,\dd B_\rho,
& \rho\in[s,\sigma],
\\[0.5ex]
X_s^{\gamma,s,\sigma,x}=x,
& x\in\R.
\end{cases}
\end{equation}
By Lemma~\ref{lem:Parisi-standard}, the drift of this diffusion process
is bounded and globally Lipschitz in $x$, uniform in $s$ and $\sigma$. Moreover,
the diffusion process has the associated transition operator family
$(P_{r,\rho}^\gamma)_{s\leq r\leq \rho\leq\sigma}$ defined by
\begin{equation}
\label{eq:def-P}
P_{r,\rho}^\gamma f(x)
\coloneqq
\E\bigl[f\bigl(X_\rho^{\gamma,r,\rho,x}\bigr)\bigr].
\end{equation}
From the definition \eqref{eq:def-P}, we immediately see the following 
properties of $(P_{r,\rho}^\gamma)_{s\leq r\leq \rho\leq\sigma}$.
\begin{property*}
Let \(0\le s\le \sigma\le1\). The family of transition operators
\((P_{r,\rho}^\gamma)_{s\le r\le \rho\le\sigma}\) satisfies the following properties.
\begin{enumerate}
    \item \label{item:P-mass} \(P_{r,\rho}^\gamma 1=1\) for every \(s\le r\le\rho\le\sigma\).
    \item \label{item:P-semigroup} The transition operator family admits the semigroup property.
    That is, for \(s\le \rho_1\le\rho_2\le\rho_3\le\sigma\),
    $
        P_{\rho_1,\rho_2}^\gamma P_{\rho_2,\rho_3}^\gamma
        =
        P_{\rho_1,\rho_3}^\gamma.
    $
\end{enumerate}
Moreover, for every bounded measurable function \(f:\R\to\R\),
\begin{enumerate}
    \setcounter{enumi}{2}
    \item \label{item:P-identity} \(P_{r,r}^\gamma f=f\) for every \(r\in[s,\sigma]\).
    \item \label{item:P-even} If \(f\) is even, then \(P_{r,\rho}^\gamma f\) is even for every
    \(s\le r\le\rho\le\sigma\).
    \item \label{item:P-positive} If \(f\ge0\), then \(P_{r,\rho}^\gamma f\ge0\) for every
    \(s\le r\le\rho\le\sigma\).
\end{enumerate}
\end{property*}

\begin{proof}
Property~\eqref{item:P-mass} and Property~\eqref{item:P-positive} follow directly from the definition.
Property~\eqref{item:P-identity} follows from
\(X_r^{\gamma,r,r,x}=x\).

Property~\eqref{item:P-semigroup} follows from the Markov property of the diffusion.
Indeed, for
\(s\le \rho_1\le\rho_2\le\rho_3\le\sigma\),
\begin{align}
P_{\rho_1,\rho_3}^\gamma f(x)
&=
\E\left[f(X_{\rho_3}^{\gamma,\rho_1,\rho_3,x})\right]
\nonumber
\\
&=
\E\left[
    \E\left[
        f(X_{\rho_3}^{\gamma,\rho_1,\rho_3,x})
        \,\middle|\,
        X_{\rho_2}^{\gamma,\rho_1,\rho_3,x}
    \right]
\right]
=
\E\left[
    P_{\rho_2,\rho_3}^\gamma f
    (X_{\rho_2}^{\gamma,\rho_1,\rho_2,x})
\right]
=
P_{\rho_1,\rho_2}^\gamma
\bigl[P_{\rho_2,\rho_3}^\gamma f\bigr](x).
\nonumber 
\end{align}

It remains to show Property~\eqref{item:P-even}.
By Lemma~\ref{lem:Parisi-standard}(2), 
the drift term $\gamma(\tau)\tau\,\partial_x\Psi^\gamma(\tau,x)$ is odd
for all $\tau\in [r,\rho]$.
Hence,
\((-X_\tau^{\gamma,r,\rho,x})_{\tau\in[r,\rho]}\)
has the same law as
\((X_\tau^{\gamma,r,\rho,-x})_{\tau\in[r,\rho]}\). Therefore, if \(f\) is even,
\[
P_{r,\rho}^\gamma f(-x)
=
\E f(X_\rho^{\gamma,r,\rho,-x})
=
\E f(-X_\rho^{\gamma,r,\rho,x})
=
\E f(X_\rho^{\gamma,r,\rho,x})
=
P_{r,\rho}^\gamma f(x).
\qedhere
\]
\end{proof}

In addition to these elementary properties of the transition family, we will use
the following standard facts without further comment.\begin{enumerate}
\item
Given terminal condition \(\varphi\in C_b^2(\mathbb R)\) and source
\(b\in C([s,\sigma];C_b^2(\mathbb R))\), the backward Cauchy problem
\begin{equation}
\begin{cases}
-\partial_\rho U(\rho,x)=\gamma(\rho)L_\rho^\gamma U(\rho,x)+b(\rho,x),
    &(\rho,x)\in[s,\sigma)\times\mathbb R,
    \\[1ex]
U(\sigma,x)=\varphi(x),
    &x\in\R
\end{cases}
\label{eq:backwardCauchy}
\end{equation}
has a unique bounded classical solution given by
the Feynman--Kac formula 
(cf. \S 5.7, Theorem 7.6 of Karatzas~and~Shreve~\cite{KaratzasShreve1991})
\begin{equation}
U(\rho,x)
=
P_{\rho,\sigma}^\gamma\varphi(x)
+
\int_\rho^\sigma P_{\rho,r}^\gamma[b(r,\cdot)](x)\dd r.
\label{eq:FK} 
\end{equation}
In particular
if \(U(\rho,x)=P_{\rho,\sigma}^\gamma\varphi(x)\), then \(\paraop_\rho u=0\) on \([s,\sigma)\times\mathbb R\).
\item
If \(f\in C_b^{1,2}([s,\sigma]\times\mathbb R)\), then for
\(s\le \rho<r\le \sigma\), we have the terminal time differentiation formula
\begin{equation}
\partial_r P_{\rho,r}^\gamma[f(r,\cdot)](x)
=
P_{\rho,r}^\gamma[
\paraop_r f(r,\cdot)](x).
\label{eq:TD} 
\end{equation}
This can be derived from applying Dynkin's formula to $(r,X_r^{\gamma,s,\sigma,x})_{r\in [s,\sigma]}$,
whose generator is $(\paraop_r)_{r\in [s,\sigma]}$.
\end{enumerate}

\subsubsection{Properties of auxiliary functions}

We now introduce two auxiliary functions that will enter the kernel representation:
\begin{equation}
v^\gamma(s,x)
=
\partial_{xx}\Psi^\gamma(s,x)
+
s\bigl(\partial_x\Psi^\gamma(s,x)\bigr)^2,
\label{eq:def-vgamma-new}
\end{equation}
and, for \(0\le \rho\le \sigma\le 1\),
\begin{equation}
w_{\rho,\sigma}^\gamma(x)
=
P_{\rho,\sigma}^\gamma
\bigl[
(\partial_x\Psi^\gamma(\sigma,\cdot))^2
\bigr](x).
\label{eq:def-wgamma-new}
\end{equation}
Now, we record a few identities for the two auxiliary functions 
that will be convenient for the computation.

The first one is a differentiation identity for $v^\gamma$.
\begin{lemma}
\label{lem:Dv-rule}
For all $(s,x)\in [0,1]\times \R$,
\begin{equation}
\paraop_s
v^\gamma(s,x)
=
(\partial_x\Psi^\gamma)^2.
\label{eq:v-gamma}
\end{equation}
Therefore, for all $0\leq \rho\leq r\leq 1$,
\begin{equation}
P_{\rho,r}^\gamma[v^\gamma(r,\cdot)](x)
=
1-\int_r^1 w_{\rho,\sigma}^\gamma(x)\,d\sigma.
\label{eq:Pv-through-w-new}
\end{equation}
\end{lemma}

\begin{proof}
The proof of \eqref{eq:v-gamma} is computational, using the following two sublemmas.
We first record a differentiation identity that follows directly from
\eqref{eq:gamma-pde}.
\begin{sublemma}
    \label{sublem:DsPsi}
We have
\[\paraop_s\partial_x\Psi^\gamma=0.\]
\end{sublemma}
\begin{proof}[Proof of Sublemma~\ref{sublem:DsPsi}]
Differentiating the time-changed Parisi PDE \eqref{eq:gamma-pde}
with respect to $x$ gives the identity.    
\end{proof}
Next, we record the differentiation rules for \(\paraop_s\).
\begin{sublemma}
\label{lem:Ds-rule}
For all $u$ and $w$ with sufficient differentiabilities,
\begin{align}
\paraop_s(uw)
&=
u\,\paraop_sw+w\,\paraop_su
+\gamma(s)\partial_xu\,\partial_xw,
\label{eq:Ds-product-rule}
\\
\paraop_s(\partial_xu)
&=
\partial_x(\paraop_su)
-
\gamma(s)s\,\partial_{xx}\Psi^\gamma(s,\cdot)\,\partial_xu,
\label{eq:Ds-first-derivative-rule}
\\
\paraop_s(\partial_{xx}u)
&=
\partial_{xx}(\paraop_su)
-2\gamma(s)s\,\partial_{xx}\Psi^\gamma(s,\cdot)\,\partial_{xx}u
-
\gamma(s)s\,\partial_{xxx}\Psi^\gamma(s,\cdot)\,\partial_xu,
\label{eq:Ds-second-derivative-rule}
\end{align}
\end{sublemma}
\begin{proof}[Proof of Sublemma~\ref{lem:Ds-rule}]
By differentiation.
\end{proof}
We now assert \eqref{eq:v-gamma}.
By Sublemma~\ref{sublem:DsPsi} and \eqref{eq:Ds-second-derivative-rule} from Sublemma~\ref{lem:Ds-rule},
\[
\mathscr D_s\partial_{xx}\Psi^\gamma
=-\gamma(s)s\bigl(\partial_{xx}\Psi^\gamma\bigr)^2.
\]
Recall from \eqref{eq:def-vgamma-new}
that $v^\gamma
=
\partial_{xx}\Psi^\gamma
+s\bigl(\partial_x\Psi^\gamma\bigr)^2$,
so
\eqref{eq:Ds-product-rule} from Sublemma~\ref{lem:Ds-rule} gives
\[
\mathscr D_s\bigl(s\bigl(\partial_x\Psi^\gamma\bigr)^2\bigr)
=
\bigl(\partial_x\Psi^\gamma\bigr)^2
+s\,\mathscr D_s\bigl(\bigl(\partial_x\Psi^\gamma\bigr)^2\bigr)
=
\bigl(\partial_x\Psi^\gamma\bigr)^2
+\gamma(s)s\bigl(\partial_{xx}\Psi^\gamma\bigr)^2.
\]
Therefore
\[
\mathscr D_s\partial_{xx}\Psi^\gamma
+\mathscr D_s\bigl(s\bigl(\partial_x\Psi^\gamma\bigr)^2\bigr)
=
-\gamma(s)s\bigl(\partial_{xx}\Psi^\gamma\bigr)^2
+\bigl(\partial_x\Psi^\gamma\bigr)^2
+\gamma(s)s\bigl(\partial_{xx}\Psi^\gamma\bigr)^2
=
\bigl(\partial_x\Psi^\gamma\bigr)^2.
\]

It remains to prove \eqref{eq:Pv-through-w-new}. Fix
\(0\leq \rho\leq r\leq 1\). 
By \eqref{eq:v-gamma}, the function $v^\gamma$
solves
\[
-\partial_\tau U(\tau,x)
=
\gamma(\tau)L_\tau^\gamma U(\tau,x)
-
\bigl(\partial_x\Psi^\gamma(\tau,x)\bigr)^2,
\qquad
u(r,x)=v^\gamma(r,x).
\]
on $(\tau,x)\in [\rho,r]\times \R$.
Then, the Feynman--Kac formula and the definition
 of \(w_{\rho,\sigma}^\gamma\)
give
\begin{equation}
v^\gamma(\rho,x)
=
P_{\rho,r}^\gamma[v^\gamma(r,\cdot)](x)
-
\int_\rho^r w_{\rho,\sigma}^\gamma(x)
\dd\sigma.
\label{eq:v-FK.1}
\end{equation}
Taking \(r=1\) in \eqref{eq:v-FK.1} and using
\[
v^\gamma(1,x)
=
\partial_{xx}\log\cosh x
+
(\partial_x\log\cosh x)^2
=
1
\]
and Property~\eqref{item:P-mass},
we obtain
\[
v^\gamma(\rho,x)
=
P_{\rho,1}^\gamma[v^\gamma(1,\cdot)](x)
-
\int_\rho^1 w_{\rho,\sigma}^\gamma(x)
\dd\sigma
=
1
-
\int_\rho^1 w_{\rho,\sigma}^\gamma(x)\,d\sigma .
\]
Substituting this into \eqref{eq:v-FK.1}
and rearranging the expressions prove
\eqref{eq:Pv-through-w-new}.
\end{proof}

Lemma~\ref{lem:Dv-rule} implies the following corollary recording 
the derivative identities for $P_{\rho,r}^\gamma[v^\gamma(r,\cdot)]$.
\begin{corollary}
\label{cor:Pv-through-w}
For every \(0\le \rho\le r\le 1\) and \(x\in\mathbb R\),
\begin{align}
\partial_r P_{\rho,r}^\gamma[v^\gamma(r,\cdot)](x)
&=
w_{\rho,r}^\gamma(x),
\label{eq:dPv-through-w-new}
\\
L_\rho^\gamma P_{\rho,r}^\gamma[v^\gamma(r,\cdot)](x)
&=
-\int_r^1 L_\rho^\gamma w_{\rho,\sigma}^\gamma(x)\,d\sigma,
\label{eq:LV-new}
\\
\partial_xP_{\rho,r}^\gamma[v^\gamma(r,\cdot)](x)
&=
-\int_r^1 \partial_xw_{\rho,\sigma}^\gamma(x)\,d\sigma.
\label{eq:dxV-new}
\end{align}
\end{corollary}
\begin{proof}
These are direct consequences of \eqref{eq:Pv-through-w-new}
from Lemma~\ref{lem:Dv-rule}.
\end{proof}

Finally, we record the identity for
\(\paraop_\rho(L_\rho^\gamma w_{\rho,\sigma}^\gamma)\)
that will be needed later.

\begin{corollary}
\label{cor:D-Lw}
For \(0\le \rho<\sigma\le 1\),
\[
\paraop_\rho\bigl(L_\rho^\gamma w_{\rho,\sigma}^\gamma\bigr)
=
\partial_x\Psi^\gamma(\rho,\cdot)\,\partial_xw_{\rho,\sigma}^\gamma
-
\frac12\gamma(\rho)\rho\,\partial_xv^\gamma(\rho,\cdot)\,
\partial_xw_{\rho,\sigma}^\gamma.
\label{eq:D-Lw-new}
\]
\end{corollary}
\begin{proof}
Let us start with recording two useful identities
for the proof.

Since \(w_{\rho,\sigma}^\gamma=
P_{\rho,\sigma}^\gamma[
(\partial_x\Psi^\gamma(\sigma,\cdot))^2]\), the Feynman--Kac formula gives
\begin{equation}
\paraop_\rho w_{\rho,\sigma}^\gamma=0.
\label{eq:Drhow}
\end{equation}
Recall from \eqref{eq:def-vgamma-new}
that $v^\gamma
=
\partial_{xx}\Psi^\gamma
+s\bigl(\partial_x\Psi^\gamma\bigr)^2$,
so differentiating and rearrange this expression yield
\begin{equation}
    \partial_{xxx}\Psi^\gamma(\rho,\cdot)
    =
    \partial_x v^\gamma - 2\rho\partial_x\Psi^\gamma(\rho,\cdot)\partial_{xx}\Psi^\gamma(\rho,\cdot).
    \label{eq:dxxxPsi}
\end{equation}

Now, we apply \(\paraop_\rho\) separately to the two terms in the identity
\[
L_\rho^\gamma w_{\rho,\sigma}^\gamma
=
\frac12\partial_{xx}w_{\rho,\sigma}^\gamma
+
\rho\,\partial_x\Psi^\gamma(\rho,\cdot)\partial_xw_{\rho,\sigma}^\gamma .
\]
Applying \eqref{eq:Ds-second-derivative-rule}~from~Sublemma~\ref{lem:Ds-rule}
yields
\begin{align}
    &
    \frac{1}{2}\paraop_\rho(\partial_{xx}w_{\rho,\sigma}^\gamma)
    \nonumber
    \\
    &=
    \frac{1}{2}
    \partial_{xx} (\paraop_\rho w_{\rho,\sigma}^\gamma)
    -
    \gamma(\rho)\rho \partial_{xx}\Psi^\gamma(\rho,\cdot)
    \partial_{xx}w_{\rho,\sigma}^\gamma
    -
    \frac{1}{2}
    \gamma(\rho)\rho \partial_{xxx}\Psi^\gamma(\rho,\cdot)\partial_x w_{\rho,\sigma}^\gamma
    \nonumber 
    \\
    &=
    -
    \gamma(\rho)\rho 
    \partial_{xx}\Psi^\gamma(\rho,\cdot)
    \partial_{xx}w_{\rho,\sigma}^\gamma
    -
    \frac{1}{2}
    \gamma(\rho)\rho 
    \partial_x v^\gamma 
    \partial_x w_{\rho,\sigma}^\gamma
    + \gamma(\rho)
    \rho^2\partial_x\Psi^\gamma(\rho,\cdot)\partial_{xx}\Psi^\gamma(\rho,\cdot)
    \partial_x w_{\rho,\sigma}^\gamma.
    \label{eq:DLw-first term}
\end{align}
Applying \eqref{eq:Ds-product-rule}~from~Sublemma~\ref{lem:Ds-rule}
yields
\begin{multline}
    \paraop_\rho
    \bigl(
    \rho 
    \partial_x\Psi^\gamma(\rho,\cdot)\partial_xw_{\rho,\sigma}^\gamma
    \bigr)
    \\
    =
    \paraop_\rho
    \bigl(
    \rho 
    \partial_x\Psi^\gamma(\rho,\cdot)
    \bigr)\partial_x w_{\rho,\sigma}^\gamma
    +
    \rho 
    \partial_x\Psi^\gamma(\rho,\cdot) 
    \paraop_\rho
    \bigl(
    \partial_xw_{\rho,\sigma}^\gamma
    \bigr)
    +
    \gamma(\rho)\rho \partial_{xx}\Psi^\gamma(\rho,\cdot)\partial_{xx} w_{\rho,\sigma}^\gamma.
    \label{eq:DLw-second-term.1}
\end{multline}
Applying \eqref{eq:Ds-product-rule}~and~\eqref{eq:Ds-first-derivative-rule}~from~Sublemma~\ref{lem:Ds-rule} yields
\begin{equation*}
    \paraop_\rho
    \bigl(
    \rho 
    \partial_x\Psi^\gamma(\rho,\cdot)
    \bigr)
    =
    \rho \paraop_\rho
    \bigl(
    \partial_x\Psi^\gamma(\rho,\cdot)
    \bigr)
    +
    \partial_x\Psi^\gamma(\rho,\cdot)
    =
    \partial_x\Psi^\gamma(\rho,\cdot),
\end{equation*}
where the second equality follows from Sublemma~\ref{sublem:DsPsi}.
Applying \eqref{eq:Ds-product-rule}~from~Sublemma~\ref{lem:Ds-rule} and \eqref{eq:Drhow}
yield
\begin{equation*}
    \paraop_\rho (\partial_x w_{\rho,\sigma}^\gamma)
    =
    -
    \gamma(\rho)\rho\partial_{xx}\Psi^\gamma(\rho,\cdot)\partial_x w_{\rho,\sigma}^\gamma.
\end{equation*}
Plugging the two identities above back to \eqref{eq:DLw-second-term.1}
yields
\begin{multline}
    \paraop_\rho
    \bigl(
    \rho 
    \partial_x\Psi^\gamma(\rho,\cdot)\partial_xw_{\rho,\sigma}^\gamma
    \bigr)
    \\
    =
    \partial_x\Psi^\gamma(\rho,\cdot)
    \partial_x w_{\rho,\sigma}^\gamma
    -
    \gamma(\rho)
    \rho^2 
    \partial_x\Psi^\gamma(\rho,\cdot) 
    \partial_{xx}\Psi^\gamma(\rho,\cdot)\partial_x w_{\rho,\sigma}^\gamma
    +
    \gamma(\rho)\rho \partial_{xx}\Psi^\gamma(\rho,\cdot)\partial_{xx} w_{\rho,\sigma}^\gamma.
    \label{eq:DLw-second-term.2}
\end{multline}
Summing \eqref{eq:DLw-first term} and \eqref{eq:DLw-second-term.2}, and canceling the common terms, concludes the proof.
\end{proof}

\subsection{Min-kernel representation}
\label{sec:min-kernel-proof}

Since the argument is computational, we first carry out the calculation and then
state Proposition~\ref{prop:ordered-hessian} at the end of the section.

Recall also the abbreviations introduced in Section~\ref{subsec:variation-equations}
of the directional derivatives:
\[
  V_h^\gamma
  =
  D\Psi^\gamma[h],
  \qquad
  W_{h_1,h_2}^\gamma
  =
  D^2\Psi^\gamma[h_1,h_2],
\]
for all \(h,h_1,h_2\in\Hoelder{\alpha/2}([0,1])\).

Corollary~\ref{cor:variation-equations} and the Feynman--Kac formula yields
\begin{equation}
    V_h^\gamma(\rho,x)
    =
    \frac{1}{2}
    \int_\rho^1
    h(\sigma)
    P_{\rho,\sigma}^\gamma
    \bigl[
    v^\gamma(\sigma,\cdot)
    \bigr](x)\dd{\sigma}
    \label{eq:V-FK} 
\end{equation}
and
\begin{multline}
    W_{h_1,h_2}^\gamma(s,x)
    =
    \int_s^1
    P_{s,\rho}^\gamma
    \Bigl[
    h_1(\rho)L_\rho^\gamma 
    V_{h_2}^\gamma(\rho,\,\cdot\,)
    \Bigr](x)\dd \rho
    \\
    +
    \int_s^1
    P_{s,\rho}^\gamma
    \Bigl[
    h_2(\rho)L_\rho^\gamma 
    V_{h_1}^\gamma(\rho,\,\cdot\,)
    \Bigr](x)
    \dd \rho
    +
    \int_s^1
    P_{s,\rho}^\gamma
    \Bigl[
    \gamma(\rho)\rho\,
    \partial_xV_{h_1}^\gamma(\rho,\,\cdot\,)
    \partial_xV_{h_2}^\gamma(\rho,\,\cdot\,)
    \Bigr](x)
    \dd \rho.
    \label{eq:W-FK}
\end{multline}

Using \eqref{eq:V-FK} and the identities from
Corollary~\ref{cor:Pv-through-w}, the integrands in the first two terms of
\eqref{eq:W-FK} are
\begin{align}
    \int_s^1
    P_{s,\rho}^\gamma
    \Bigl[
    h_1(\rho)L_\rho^\gamma 
    V_{h_2}^\gamma (\rho,\,\cdot\,)
    \Bigr](x)
    \dd{\rho}
    &=
    -
    \frac12
    \int_s^1
    \int_{\Delta_{s,\sigma}}
    h_1(\rho)h_2(r)
    P_{s,\rho}^\gamma 
    \bigl[
    L_\rho^\gamma 
    w_{\rho,\sigma}^\gamma
    \bigr]
    (x)
    \dd{r}
    \dd{\rho}
    \dd{\sigma},
    \label{eq:first-term}
    \\
    \int_s^1
    P_{s,\rho}^\gamma
    \Bigl[
    h_2(\rho)L_\rho^\gamma 
    V_{h_1}^\gamma (\rho,\,\cdot\,)
    \Bigr](x)
    \dd{\rho}
    &=
    -
    \frac12
    \int_s^1
    \int_{\Delta_{s,\sigma}}
    h_2(\rho)h_1(r)
    P_{s,\rho}^\gamma 
    \bigl[
    L_\rho^\gamma 
    w_{\rho,\sigma}^\gamma
    \bigr]
    (x)
    \dd{r}
    \dd{\rho}
    \dd{\sigma},
    \label{eq:second-term}
\end{align}
with $\Delta_{s,\sigma}=\bigl\{(\rho,r) \,\big|\, s\leq \rho\leq r\leq \sigma\bigr\}$
for $\sigma\in [s,1]$.
By Fubini and
adopting the change of variables $\rho=\rho_1$ and $r=\rho_2$ for \eqref{eq:first-term}
and the change of variables $\rho=\rho_2$ and $r=\rho_1$ for \eqref{eq:second-term},
the sum of \eqref{eq:first-term} and \eqref{eq:second-term} equals
\begin{equation}
    -
    \frac12
    \int_s^1
    \iint_{[s,\sigma]^2}
    h_1(\rho_1)h_2(\rho_2)
    P_{s,\rho_1\wedge\rho_2}^\gamma 
    \bigl[
    L_{\rho_1\wedge \rho_2}^\gamma 
    w_{\rho_1\wedge \rho_2,\sigma}^\gamma
    \bigr]
    (x)
    \dd{\rho_1}
    \dd{\rho_2}
    \dd{\sigma}.
    \label{eq:first-second-answer}
\end{equation}

We now treat the third term in \eqref{eq:W-FK}.
Using \eqref{eq:V-FK}, the third term in \eqref{eq:W-FK} equals
\begin{multline}
    \int_s^1
    P_{s,\rho}^\gamma
    \Bigl[
    \gamma(\rho)\rho\,
    \partial_xV_{h_1}^\gamma(\rho,\,\cdot\,)
    \partial_xV_{h_2}^\gamma(\rho,\,\cdot\,)
    \Bigr](x)
    \dd{\rho}
    \\
    =
    \frac14
    \int_s^1
    \iint_{[\rho,1]^2}
    h_1(\rho_1)h_2(\rho_2)
    P_{s,\rho}^\gamma
    \Bigl[
    \gamma(\rho)\rho\,
    \partial_x 
    P_{\rho,\rho_1}^\gamma 
    \bigl[
    v^\gamma(\rho_1,\,\cdot\,)
    \bigr]
    \partial_x 
    P_{\rho,\rho_2}^\gamma
    \bigl[
    v^\gamma(\rho_2,\,\cdot\,)
    \bigr]
    \Bigr](x)
    \dd{\rho_1}\dd{\rho_2}
    \dd{\rho}.
    \label{eq:quadratic-term}
\end{multline} 
The domain in \eqref{eq:quadratic-term} admits another parametrization
\begin{equation*}
    \Bigl\{
        (\rho,\rho_1,\rho_2) 
        \,\Big|\, 
        s\leq \rho\leq 1,\, (\rho_1,\rho_2)\in [\rho,1]^2
    \Bigr\}
    =
    \Bigl\{
        (\rho,\rho_1,\rho_2) 
        \,\Big|\, 
        s\leq \rho\leq \rho_1\wedge\rho_2,\, (\rho_1,\rho_2)\in [s,1]^2
    \Bigr\},
\end{equation*}
so Fubini yields
\begin{multline}
    \eqref{eq:quadratic-term}
    =
    \frac14
    \iint_{[s,1]^2}
    h_1(\rho_1)h_2(\rho_2)
    \\
    \int_s^{\rho_1\wedge\rho_2}
    P_{s,\rho}^\gamma
    \Bigl[
    \gamma(\rho)\rho\,
    \partial_x 
    P_{\rho,\rho_1}^\gamma 
    \bigl[
    v^\gamma(\rho_1,\,\cdot\,)
    \bigr]
    \partial_x 
    P_{\rho,\rho_2}^\gamma
    \bigl[
    v^\gamma(\rho_2,\,\cdot\,)
    \bigr]
    \Bigr](x)
    \dd{\rho}
    \dd{\rho_1}\dd{\rho_2}.
    \label{eq:quadratic-term-reparametrization}
\end{multline}
Applying Corollary~\ref{cor:Pv-through-w} with
\(r=\rho_1\vee\rho_2\), the inner integral in
\eqref{eq:quadratic-term-reparametrization} becomes
\begin{multline}
    \int_s^{\rho_1\wedge\rho_2}
    P_{s,\rho}^\gamma
    \Bigl[
    \gamma(\rho)\rho\,
    \partial_x 
    P_{\rho,\rho_1\wedge\rho_2}^\gamma 
    \bigl[
    v^\gamma(\rho_1\wedge\rho_2,\,\cdot\,)
    \bigr]
    \partial_x 
    P_{\rho,\rho_1\vee\rho_2}^\gamma
    \bigl[
    v^\gamma(\rho_1\vee\rho_2,\,\cdot\,)
    \bigr]
    \Bigr](x)
    \dd{\rho}
    \\
    =
    -
    \int_s^{\rho_1\wedge\rho_2}
    \int_{\rho_1\vee\rho_2}^1
    P_{s,\rho}^\gamma
    \Bigl[
    \gamma(\rho)\rho\,
    \partial_x 
    P_{\rho,\rho_1\wedge\rho_2}^\gamma 
    \bigl[
    v^\gamma(\rho_1\wedge\rho_2,\,\cdot\,)
    \bigr]
    \partial_x 
    w_{\rho,\sigma}^\gamma
    \Bigr](x)
    \dd{\sigma}
    \dd{\rho}.
    \label{eq:inner-integral}
\end{multline}
Plug \eqref{eq:inner-integral} back to \eqref{eq:quadratic-term}.
Then, applying Fubini with the reparametrization
\[
    \rho_1,\rho_2\in[s,1],\ \rho_1\vee\rho_2\le\sigma\le1
    \quad\text{if and only if}\quad
    s\le\sigma\le1,\ \rho_1,\rho_2\in[s,\sigma],
\]
yields
\begin{multline}
    \eqref{eq:quadratic-term} 
    =
    -
    \frac14
    \int_s^1
    \iint_{[s,\sigma]^2}
    h_1(\rho_1)h_2(\rho_2)
    \\
    \int_s^{\rho_1\wedge\rho_2}
    P_{s,\rho}^\gamma
    \Bigl[
    \gamma(\rho)\rho\,
    \partial_x 
    P_{\rho,\rho_1\wedge\rho_2}^\gamma 
    \bigl[
    v^\gamma(\rho_1\wedge\rho_2,\,\cdot\,)
    \bigr]
    \partial_x 
    w_{\rho,\sigma}^\gamma
    \Bigr](x)
    \dd{\rho}
    \dd{\rho_1}\dd{\rho_2}
    \dd{\sigma}.
    \label{eq:third-term-answer}
\end{multline}

Combining \eqref{eq:first-second-answer} and \eqref{eq:third-term-answer}
yields the following min-kernel representation,
stated as a proposition.
\begin{proposition}
\label{prop:ordered-hessian}
Let \(h_1,h_2\in \Hoelder{\alpha/2}([0,1])\).
Then, for all $(s,x)\in [0,1]\times \R$,
\begin{equation*}
D^2\Psi^\gamma[h_1,h_2](s,x)
=
-
\int_s^1
\iint_{[s,\sigma]^2}
h_1(\rho_1)h_2(\rho_2)
A^{\gamma,s,\sigma,x}(\rho_1\wedge\rho_2)
\dd\rho_1\dd\rho_2\dd{\sigma},
\end{equation*}
where
\begin{equation}
    \label{eq:min-kernel-FK}
    A^{\gamma,s,\sigma,x}(\rho)
    =
    \frac12
    P_{s,\rho}^\gamma[L_\rho^\gamma w_{\rho,\sigma}^\gamma](x)
    +
    \frac{1}{4}
    \int_s^\rho
    P_{s,\tau}^\gamma
    \Bigl[
        \gamma(\tau)\tau 
        \partial_x P_{\tau,\rho}^\gamma[v^\gamma(\rho,\,\cdot\,)]
        \partial_x w_{\tau,\sigma}^\gamma
    \Bigr](x)
    \dd{\tau}
\end{equation}
with $\rho\in [s,\sigma]$.
\end{proposition}

\subsection{Corollaries of the min-kernel representation}
\label{sec:cor-min-kernel}

We now state a few corollaries of Proposition~\ref{prop:ordered-hessian}
that will be useful in the next sections.

The first one is a factorization of the min-kernel representation.

\begin{corollary}
\label{cor:min-kernel-factorization}
Fix \(h_1,h_2\in \Hoelder{\alpha/2}([0,1])\).
Then, for all $(s,x)\in [0,1]\times \R$,
\begin{multline*}
D^2\Psi^\gamma[h_1,h_2](s,x)
=
-
\int_s^1
\biggl(
A^{\gamma,s,\sigma,x}(s)
\Bigl(
\iint_{[s,\sigma]^2}
h_1(\rho_1)h_2(\rho_2)
\dd{\rho_1}\dd{\rho_2}
\Bigr)
\\
+
\int_s^\sigma
\Bigl(
\iint_{[\rho,\sigma]^2}
h_1(\rho_1)h_2(\rho_2)
\dd\rho_1\dd\rho_2
\Bigr)
\partial_\rho A^{\gamma,s,\sigma,x}(\rho)
\dd{\rho}
\biggr)
\dd{\sigma}.
\end{multline*}
\end{corollary}
\begin{proof}
The statement follows from the factorization
\begin{equation*}
\label{eq:stieltjes-min-representation}
A^{\gamma,s,\sigma,x}(\rho_1\wedge\rho_2)
=
A^{\gamma,s,\sigma,x}(s)
+
\int_{(s,\sigma]}
\indic_{\{\rho\le\rho_1\}}
\indic_{\{\rho\le\rho_2\}}
\,
\partial_\rho A^{\gamma,s,\sigma,x}(\rho)
\dd{\rho}
\end{equation*}
where \(\rho_1,\rho_2\in[s,\sigma]\).
\end{proof}

By Corollary~\ref{cor:min-kernel-factorization}, the sign of the second
Fr\'echet derivative \(D^2\Psi^\gamma\) is governed by the left endpoint
\(A^{\gamma,s,\sigma,x}(s)\) and the density
\(\partial_\rho A^{\gamma,s,\sigma,x}(\rho)\).
The next corollary reduces the problem to studying the auxiliary function
$w_{s,\sigma}^\gamma$, which will be done in the next section.

\begin{corollary}
    \label{cor:A-data}
    Fix \(0\le s<\sigma\le 1\).
    Then, the following are true.
    \begin{enumerate}
        \item For all $x\in\R$,
        \begin{equation}
        A^{\gamma,s,\sigma,x}(s)
        =
        \frac{1}{4}\partial_{xx}w_{s,\sigma}^\gamma(x)
        +
        \frac{s}{2}\partial_x\Psi^\gamma(s,x)\partial_x w_{s,\sigma}^\gamma(x), 
        \label{eq:left-point}
        \end{equation}
        \item For all \(\rho\in(s,\sigma)\), and \(x\in\mathbb R\),
        \begin{equation}
        \partial_\rho A^{\gamma,s,\sigma,x}(\rho)
        =
        \frac12
        P_{s,\rho}^\gamma
        \Bigl[
        \partial_x\Psi^\gamma(\rho,\cdot)\,
        \partial_x w_{\rho,\sigma}^\gamma
        \Bigr](x)
        +
        \frac14
        \int_s^\rho
        P_{s,\tau}^\gamma
        \Bigl[
        \gamma(\tau)\tau\,
        \partial_x w_{\tau,\rho}^\gamma\,
        \partial_x w_{\tau,\sigma}^\gamma
        \Bigr](x)\,\dd\tau .
        \label{eq:kernel-density-FK}
        \end{equation}
    \end{enumerate}
\end{corollary}
\begin{proof}
Recall the definition of \(A^{\gamma,s,\sigma,x}\) in
\eqref{eq:min-kernel-FK}:
\begin{equation}
A^{\gamma,s,\sigma,x}(\rho)
    =
    \frac12
    P_{s,\rho}^\gamma[L_\rho^\gamma w_{\rho,\sigma}^\gamma](x)
    +
    \frac{1}{4}
    \int_s^\rho
    P_{s,\tau}^\gamma
    \Bigl[
        \gamma(\tau)\tau 
        \partial_x P_{\tau,\rho}^\gamma[v^\gamma(\rho,\,\cdot\,)]
        \partial_x w_{\tau,\sigma}^\gamma
    \Bigr](x)
    \dd{\tau}.
    \nonumber
\end{equation}

We first prove \eqref{eq:left-point}. Taking \(\rho=s\) in
\eqref{eq:min-kernel-FK}, the integral over \([s,s]\) vanishes. Moreover,
Property~\eqref{item:P-identity} gives \(P_{s,s}^\gamma=\mathrm{Id}\). Therefore,
we obtain
\[
    A^{\gamma,s,\sigma,x}(s)
    =
    \frac12 L_s^\gamma w_{s,\sigma}^\gamma(x)
    =
    \frac{1}{4}\partial_{xx}w_{s,\sigma}^\gamma(x)
        +
        \frac{s}{2}\partial_x\Psi^\gamma(s,x)\partial_x w_{s,\sigma}^\gamma(x),
\]
where the second equality follows from the definition of \(L_s^\gamma\) in
\eqref{eq:def-Lgamma}.

It remains to prove \eqref{eq:kernel-density-FK}. Applying the endpoint
differentiation formula \eqref{eq:TD} and Corollary~\ref{cor:D-Lw}
to the first term in \eqref{eq:min-kernel-FK} gives
\begin{align}
    \partial_\rho
    \frac12
    P_{s,\rho}^\gamma
    \bigl[
        L_\rho^\gamma w_{\rho,\sigma}^\gamma
    \bigr](x)
    &=
    \frac12
    P_{s,\rho}^\gamma
    \Bigl[
        \paraop_\rho
        \bigl(
            L_\rho^\gamma w_{\rho,\sigma}^\gamma
        \bigr)
    \Bigr](x)
    \nonumber\\
    &=
    \frac12
    P_{s,\rho}^\gamma
    \Bigl[
        \partial_x\Psi^\gamma(\rho,\cdot)
        \partial_xw_{\rho,\sigma}^\gamma
    \Bigr](x)
    -
    \frac14
    P_{s,\rho}^\gamma
    \Bigl[
        \gamma(\rho)\rho\,
        \partial_xv^\gamma(\rho,\cdot)
        \partial_xw_{\rho,\sigma}^\gamma
    \Bigr](x).
    \label{eq:first-kernel-density-term}
\end{align}

We now treat the second term in \eqref{eq:min-kernel-FK}.
Note that 
\begin{align}
\partial_x P_{\rho,\rho}^\gamma[v^\gamma(\rho,\cdot)]
&=
\partial_x v^\gamma(\rho,\cdot),
\nonumber\\
\partial_\rho
\partial_xP_{\tau,\rho}^\gamma[v^\gamma(\rho,\cdot)]
&=
\partial_xw_{\tau,\rho}^\gamma,
\nonumber
\end{align}
where the first identity follows from Property~\eqref{item:P-identity}
(\(P_{\rho,\rho}^\gamma=\mathrm{Id}\)), and the second follows from
\eqref{eq:dPv-through-w-new} in Corollary~\ref{cor:Pv-through-w}. We now
differentiate the second term in \eqref{eq:min-kernel-FK}.
By the Leibniz integral rule and the two identities
above, this equals
\begin{equation}
    \frac14
    P_{s,\rho}^\gamma
    \Bigl[
        \gamma(\rho)\rho\,
        \partial_xv^\gamma(\rho,\cdot)
        \partial_xw_{\rho,\sigma}^\gamma
    \Bigr](x)
    +
    \frac14
    \int_s^\rho
    P_{s,\tau}^\gamma
    \Bigl[
        \gamma(\tau)\tau\,
        \partial_xw_{\tau,\rho}^\gamma\,
        \partial_xw_{\tau,\sigma}^\gamma
    \Bigr](x)\,\dd\tau .
    \label{eq:second-kernel-density-term}
\end{equation}
The first term in \eqref{eq:second-kernel-density-term} cancels the second term
in \eqref{eq:first-kernel-density-term}, 
so summing
\eqref{eq:first-kernel-density-term} and \eqref{eq:second-kernel-density-term}
yields \eqref{eq:kernel-density-FK}.
\end{proof}

\section{Cone preservation of the transition kernel}
\label{sec:cone-preservation}

As explained in Section~\ref{sec:cor-min-kernel}, 
the sign of the second
Fr\'echet derivative \(D^2\Psi^\gamma\) is governed by the left endpoint
\(A^{\gamma,s,\sigma,x}(s)\) and the density
\(\partial_\rho A^{\gamma,s,\sigma,x}(\rho)\).
In view of Corollary~\ref{cor:A-data}, this reduces to understanding the function
\begin{equation}
    w_{\rho,\sigma}^\gamma(x)
    =
    P_{\rho,\sigma}^\gamma\bigl[
    (\partial_x\Psi^\gamma(\sigma,\cdot))^2
    \bigr](x).
\end{equation}
By Lemma~\ref{lem:Parisi-standard}(2)--(3), $x\mapsto (\partial_x\Psi^\gamma(\sigma,x))^2$
is even and nonnegative on $\R$ and nondecreasing on $[0,\infty)$,
which motivates the cone-preservation property of the transition operator
\(P_{\rho,\sigma}^\gamma\). More precisely, define the cone
\begin{equation}
\cone
=
\Bigl\{
 f\in C_b^{3+\alpha}(\mathbb R)
 \,\Big|\,
 f \text{ is even, nonnegative, and nondecreasing on }[0,\infty)
\Bigr\}.
\label{eq:def-cone}
\end{equation}
The purpose of this section is to prove that the transition operator
preserves the cone.

\begin{proposition}
\label{prop:cone-preservation}
Fix \(0\le \rho\le \sigma\le1\).
If \(f\in \cone\), then $P_{\rho,\sigma}^\gamma f\in\mathcal C$.
\end{proposition}

\begin{remark}
\label{rem:Auffinger-Chen-proof}
The auxiliary function \(w_{\rho,\sigma}^\gamma\) is the variance term in
Proposition~4 in Auffinger and Chen \cite{AuffingerChen16}.
In their proof, an FKG-type argument propagates the relevant
evenness and monotonicity properties through the finite-step
Parisi recursion. In the
present formulation, the same role is played by Proposition~\ref{prop:cone-preservation}:
the transition operator \(P_{\rho,\sigma}^\gamma\) preserves the cone of even,
nonnegative functions that are nondecreasing on \([0,\infty)\). 
\end{remark}

\begin{proof}[Proof of Proposition~\ref{prop:cone-preservation}]
    Fix $f\in\cone$.

    Fix $\sigma\in [0,1]$. For all $\rho\in [0,\sigma]$ and $x\in\R$, we adopt the abbreviation
 \[
 u(\rho,x) = P_{\rho,\sigma}^\gamma f(x).
 \]
Then, Property~\eqref{item:P-even} and Property~\eqref{item:P-positive}
of the transition operator imply that 
$u(\rho,\cdot)$ is even and nonnegative.

We now verify the regularity condition. 
Note that \(u\) solves
\[
    -\partial_\rho u=\gamma(\rho)L_\rho^\gamma u,
    \qquad
    u(\sigma,\cdot)=f,
\]
on $[0,\sigma]\times\R$, so Theorem~\ref{thm:terminal-time-schauder} yields that $u\in C_b^{1+\alpha/2,2+\alpha}([0,\sigma]\times\mathbb R)$.

We adopt the abbreviation
\[
    z(\rho,x)=\partial_xu(\rho,x).
\]
Let $a_0(\rho,x)=\gamma(\rho)\rho\,\partial_{xx}\Psi^\gamma(\rho,x)$. On $[0,\sigma]\times\R$,
$z$ solves 
\[
\bigl(-\partial_\rho-\gamma(\rho)L_\rho^\gamma-a_0(\rho,x)\bigr)z=0,
\qquad
z(\sigma,x)=f'(x).
\]
Since \(f\in C_b^{3+\alpha}(\mathbb R)\), we have
\(f'\in C_b^{2+\alpha}(\mathbb R)\). Moreover, the operator satisfies the conditions for 
Theorem~\ref{thm:terminal-time-schauder}, which yields $z\in C_b^{1+\alpha/2,2+\alpha}([0,\sigma]\times\mathbb R)$
and thus the desired regularity.

It remains to prove the monotonicity on the half-line \([0,\infty)\).
The evenness of \(u(\rho,\cdot)\) gives \(z(\rho,0)=0\).
Moreover, at the terminal time \(\sigma\), we have
\[
    z(\sigma,x)
    =
    \partial_xu(\sigma,x)
    =
    f'(x)\ge0,
    \qquad x\in[0,\infty).
\]
Thus, \(z\) solves, on \((0,\sigma)\times(0,\infty)\),
\[
    \begin{cases}
    \bigl(-\partial_\rho-\gamma(\rho)L_\rho^\gamma-a_0(\rho,x)\bigr)z=0,
    & (\rho,x)\in [0,\sigma)\times[0,\infty),
    \\[0.5ex]
    z(\rho,0)=0,
    & \rho\in [0,\sigma],
    \\[0.5ex]
    z(\sigma,x)\geq 0. 
    & x\in[0,\infty).
    \end{cases}
\]
By
Lemma~\ref{lem:Parisi-standard}, \(a_0\ge0\) and is bounded.  Applying
the weak maximum principle
(Theorem~\ref{thm:comparison}) gives
\[
    z(\rho,x)\ge0,
    \qquad (\rho,x)\in [0,\sigma]\times[0,\infty).
\]
In particular, this implies \(P_{\rho,\sigma}^\gamma f=u(\rho,\cdot)\) is nondecreasing on
\([0,\infty)\), completing the proof.
\end{proof}

We now apply Corollary~\ref{cor:A-data} and Proposition~\ref{prop:cone-preservation}
to obtain the following result, which will immediately 
imply Theorem~\ref{thm:concavity}(2) and Theorem~\ref{thm:concavity}(3).

\begin{corollary}
\label{cor:A-positivity}
Fix $0\leq s< \sigma\leq 1$. Then, the following are true. 
\begin{enumerate}
    \item $A^{\gamma,s,\sigma,0}(s)\geq 0$.
    \item For all $x\in \R$ and $\rho\in (s,\sigma)$, $\partial_\rho A^{\gamma,s,\sigma,x}(\rho)\geq 0$.
\end{enumerate}
\end{corollary}

\begin{proof}
    Fix $0\leq s< \sigma\leq 1$.
    By Lemma~\ref{lem:Parisi-standard}(2)--(3), $\Psi^\gamma(\rho,\cdot)$ is convex and even, 
so \(\partial_x\Psi^\gamma(\rho,\cdot)\) is odd and nonnegative on \([0,\infty)\).
    In particular, $(\partial_x\Psi^\gamma(\rho,\cdot))^2\in\cone$.
    Since $(\partial_x\Psi^\gamma(\sigma,\cdot))^2\in\cone$, Proposition~\ref{prop:cone-preservation} yields
\[
    w_{\rho,\sigma}^\gamma
    =
    P_{\rho,\sigma}^\gamma\bigl[
    (\partial_x\Psi^\gamma(\sigma,\cdot))^2
    \bigr]
    \in\cone,
    \qquad \rho\in [s,\sigma].
\]
In particular, 
for all $\rho\in [s,\sigma], $\(\partial_xw_{\rho,\sigma}^\gamma\) is odd and nonnegative on
\([0,\infty)\).

\begin{enumerate}
\item By \eqref{eq:left-point} from Corollary~\ref{cor:A-data}
and the fact that $\partial_x\Psi^\gamma(s,\cdot)$ is odd recalled in the first paragraph,
\[
    A^{\gamma,s,\sigma,0}(s)
    =
    \frac14\partial_{xx}w_{s,\sigma}^\gamma(0)
    +
    \frac{1}{2}\, s\partial_x\Psi^\gamma(s,0)\partial_xw_{s,\sigma}^\gamma(0)
    =
    \frac14\partial_{xx}w_{s,\sigma}^\gamma(0)
\]
Since \(w_{s,\sigma}^\gamma\in\cone\)
 by the first paragraph, \(0\) is a minimum of
\(w_{s,\sigma}^\gamma\). Therefore,
\[
    A^{\gamma,s,\sigma,0}(s)
    =
    \frac14\partial_{xx}w_{s,\sigma}^\gamma(0)
    \ge0.
\]
\item Fix $x\in\R$ and \(\rho\in(s,\sigma)\). Recall that \eqref{eq:kernel-density-FK} from Corollary~\ref{cor:A-data} gives
\begin{equation*}
    \partial_\rho 
    A^{\gamma,s,\sigma,x}
    (\rho)
    =
    \frac12
    P_{s,\rho}^\gamma
    \Bigl[
        \partial_x\Psi^\gamma(\rho,\cdot)\,
        \partial_xw_{\rho,\sigma}^\gamma
    \Bigr](x)
    +
    \frac14
    \int_s^\rho
    P_{s,\tau}^\gamma
    \Bigl[
        \gamma(\tau)\tau\,
        \partial_xw_{\tau,\rho}^\gamma\,
        \partial_xw_{\tau,\sigma}^\gamma
    \Bigr](x)\,\dd\tau .
\end{equation*}
By the first paragraph, 
the functions \(\partial_x\Psi^\gamma(\rho,\cdot)\),
\(\partial_xw_{\rho,\sigma}^\gamma\), \(\partial_xw_{\tau,\rho}^\gamma\), and
\(\partial_xw_{\tau,\sigma}^\gamma\) 
are also odd and nonnegative on \([0,\infty)\).
Hence, the products inside the transition operators are
nonnegative on \(\R\), 
which implies that $\partial_\rho A^{\gamma,s,\sigma,x}(\rho)\ge0$
by 
the properties of the transition operator of \(P^\gamma\).
\qedhere
\end{enumerate}
\end{proof}

\subsection{Proof of Theorem~\ref{thm:concavity}(2)
and Theorem~\ref{thm:concavity}(3)
}
\label{sec:concavity-proof}

Let \(h_1,h_2\in \Hoelder{\alpha/2}([0,1])\).
Recall from Corollary~\ref{cor:min-kernel-factorization}
that 
\begin{multline}
D^2\Psi^\gamma[h_1,h_2](s,0)
=
-
\int_s^1
\biggl(
A^{\gamma,s,\sigma,0}(s)
\Bigl(
\iint_{[s,\sigma]^2}
h_1(\rho_1)h_2(\rho_2)
\dd{\rho_1}\dd{\rho_2}
\Bigr)
\\
+
\int_s^\sigma
\Bigl(
\iint_{[\rho,\sigma]^2}
h_1(\rho_1)h_2(\rho_2)
\dd\rho_1\dd\rho_2
\Bigr)
\partial_\rho A^{\gamma,s,\sigma,0}(\rho)
\dd{\rho}
\biggr)
\dd{\sigma}.
\label{eq:DPsi-formula}
\end{multline}
Moreover, Corollary~\ref{cor:A-positivity} implies that 
both $A^{\gamma,s,\sigma,0}(s)$ and $\partial_\rho A^{\gamma,s,\sigma,x}(\rho)
$ are nonnegative.

If \(h_1=h_2=h\), then
\eqref{eq:DPsi-formula} is nonpositive, which proves Theorem~\ref{thm:concavity}(2). 

If $h_1\geq 0$ and $h_2\geq 0$, then \eqref{eq:DPsi-formula} is nonpositive, which proves Theorem~\ref{thm:concavity}(3).

\section{Proof of Theorem~\ref{thm:chen-initial-convexity}}
\label{sec:proof-chen-initial-convexity}

We recall the path spaces \(\PathSpace_p\) from
\eqref{eq:hopfapp-scalar-path-space}.
\begin{lemma}
\label{lem:PathSpace-alpha-density}
For every \(1\le p<\infty\), the set \(\SmoothCore_\alpha\) is dense in
\(\PathSpace_p\) with respect to the \(L^p\)-norm.
\end{lemma}

\begin{proof}
Let \(\sfq\in\PathSpace_p\).  Choose a nondecreasing representative.  For
\(M>0\), set \(\sfq^M=\sfq\wedge M\).  Then
\(\sfq^M\to\sfq\) in \(L^p\) as \(M\to\infty\).  For fixed \(M\) and
\(\delta\in(0,1)\), define
\[
    \sfq^{M,\delta}(t)
    =
    \begin{cases}
    \dfrac{t}{\delta}\sfq^M(t), & 0\le t\le\delta,\\[1ex]
    \sfq^M(t), & \delta<t<1.
    \end{cases}
\]
Then \(\sfq^{M,\delta}\) is nonnegative, nondecreasing, starts from zero, and
\(\sfq^{M,\delta}\to\sfq^M\) in \(L^p\) as \(\delta\downarrow0\).
Extend \(\sfq^{M,\delta}\) to \([0,1]\) by its left limit at \(1\), and let
\(B_n\sfq^{M,\delta}\) be its Bernstein polynomial.  Since Bernstein
polynomials preserve monotonicity and endpoint values,
\(B_n\sfq^{M,\delta}\) is nonnegative, nondecreasing, and vanishes at zero.
Moreover,
\[
    B_n\sfq^{M,\delta}\longrightarrow \sfq^{M,\delta}
    \qquad\text{in }L^p([0,1]),
\]
by pointwise convergence at continuity points of the bounded monotone function
\(\sfq^{M,\delta}\) and dominated convergence.  Finally, for
\(\varepsilon>0\), set
\[
    \sfq_{M,\delta,n,\varepsilon}(t)
    =
    B_n\sfq^{M,\delta}(t)+\varepsilon t .
\]
Then \(\sfq_{M,\delta,n,\varepsilon}\in C^\infty([0,1])\), starts from zero,
and satisfies
\[
    \frac{d}{dt}\sfq_{M,\delta,n,\varepsilon}(t)
    \ge \varepsilon>0 .
\]
Thus \(\sfq_{M,\delta,n,\varepsilon}\in\SmoothCore_\alpha\).  Choosing
successively \(M\), \(\delta\), \(n\), and \(\varepsilon\) proves the density.
\end{proof}

We next prove the \(L^1\)-stability of $\psi$.

\begin{lemma}
\label{lem:smooth-core-L1-stability}
For every \(\sfq_0,\sfq_1\in\SmoothCore_\alpha\),
\[
    |\psi(\sfq_0)-\psi(\sfq_1)|
    \le
    \norm*{\sfq_0-\sfq_1}_{L^{1}[0,1]}.
\]
\end{lemma}

\begin{proof}
We adopt the following change of variable to put $\Phi^{\sfq_0}$ and 
$\Phi^{\sfq_1}$ on the common terminal time. 
Let $T=2(\sfq_0(1)\vee\sfq_1(1))$.
For all $i=0,1$, define
\begin{equation}
U_i(t,x)
=
\begin{cases}
\Phi^{2\sfq_i}(t,x)+\dfrac{T}{2}-\sfq_i(1), 
& 
(t,x)\in [0,2\sfq_i(1)]\times\R,
\\[1ex]
\log\cosh x+\dfrac{T-t}{2}, 
& (t,x)\in [2\sfq_i(1),T]\times\R.
\end{cases}
\label{eq:Ui-def}
\end{equation}
Then, \eqref{eq:Ui-def} yields
\begin{equation}
    \psi(\sfq^i)
    =
    \sfq^i(1)-\Phi^{2\sfq_i}(0,0)
    =
    \frac T2-U_i(0,0).
    \label{eq:psi-common-time}
\end{equation}
We next identify Parisi equation solved by \(U^i\).
For all $i=0,1$, define the extended inverse profiles \(\zeta_i:[0,T]\to[0,1]\) by
\[
    \zeta_i(t)
    =
    \begin{cases}
    \sfq_i^{-1}(t/2), & 0\le t\le 2\sfq_i(1),
    \\[1ex]
    1, & 2\sfq_i(1)<t\le T.
    \end{cases}
\]
Note that $\zeta_i$ is continuous and bounded.
Then \(U_i\) solves
\begin{equation}
\begin{cases}
-\partial_tU_i(t,x)
=
\frac12
\Bigl(
\partial_{xx}U_i(t,x)
+
\zeta_i(t)
\bigl(
    \partial_xU_i(t,x)
\bigr)^2
\Bigr), & (t,x)\in [0,T)\times\R,
\\[1ex]
U_i(T,x)=\log\cosh x,
& x\in\R.
\end{cases}
\label{eq:Ui-PDE}
\end{equation}
Define the operator 
\[
    \LinearOp_t
    =
    \frac12\partial_{xx}
    +
    \frac12\zeta_0(t)
    \bigl(\partial_xU_0+\partial_xU_1\bigr)\partial_x .
\]
and the coefficient
\[
    c(t,x)
    =
    \frac12
    \bigl(\zeta_0(t)-\zeta_1(t)\bigr)
    \bigl(\partial_xU_1(t,x)\bigr)^2.
\]
Subtracting the two equations in \eqref{eq:Ui-PDE} yields that difference $\tilde U = U_0-U_1$ satisfies the PDE
\[
\begin{cases}
-\partial_t\tilde U(t,x)
=
\LinearOp_t\tilde U(t,x)
+
c(t,x), 
&
(t,x)\in [0,T]\times\R,
\\[1ex]
\tilde U(T,x)=0,
&
x\in\R.
\end{cases}
\]
Define
\[
    R(t)
    =
    \frac12
    \int_t^T |\zeta_0(r)-\zeta_1(r)|\dd r,
    \qquad t\in [0,T].
\]
We now prove, by the weak maximum principle, that
\begin{equation}
    \sup_{x\in\R}\abs*{\tilde U(t,x)}\leq R(t),
    \qquad t\in [0,T].
\end{equation}
By Lemma~\ref{lem:parisi-schauder-class} and \eqref{eq:Ui-def},
\(U_i-\log\cosh\), \(i=0,1\), are bounded, so \(\tilde U\) is bounded.
Also, by Lemma~\ref{lem:Parisi-standard}(3), $|\partial_xU_i|\le1$ for $i=0,1$.
Thus
\[
    |c(t,x)|
    \le
    \frac12|\zeta^0(t)-\zeta^1(t)|.
\]
For the operator $\LinearOp_t$,
the coefficient of \(\partial_{xx}\) is \(1/2\), and the drift coefficient is bounded and
continuous because \(0\le \zeta_0\le1\) and \(|\partial_xU^i|\le1\), \(i=0,1\).
Therefore Theorem~\ref{thm:comparison}
applies on \([0,T]\times\mathbb R\), with \(a_0=0\). Since
\[
    (-\partial_t-\LinearOp_t)(R\pm\tilde U)
    =
    \frac12|\zeta_0-\zeta_1|\pm c
    \ge0,
    \qquad
    (R\pm \tilde U)(T,\cdot)=0,
\]
we get \(-R\leq \tilde U\le R\). 
Hence
\[
    |\tilde U(0,0)|
    \le
    \frac12\int_0^T |\zeta^0(t)-\zeta^1(t)|\dd t .
\]
Finally, using the identity
\[
    \int_0^T |\zeta^0(t)-\zeta^1(t)|\dd t
    =
    2\int_0^1 |\sfq_0(s)-\sfq_1(s)|\dd s,
\]
we conclude that
\[
|\psi(\sfq^0)-\psi(\sfq^1)|
=
|\tilde U(0,0)|
\le
\int_0^1 |\sfq_0(s)-\sfq_1(s)|\dd s.
\]
This proves the lemma.
\end{proof}

We now prove Theorem~\ref{thm:chen-initial-convexity}.

\begin{proof}[Proof of Theorem~\ref{thm:chen-initial-convexity}]
Using \eqref{eq:gamma-change-of-variables} with \(2\sfq\) in place of
\(\sfq\),
\begin{equation}
\label{eq:chen-initial-smooth-representation}
    \psi(\sfq)
    =
    \sfq(1)-\Phi^{2\sfq}(0,0)
    =
    \sfq(1)-\Psi^{2\dot\sfq}(0,0), 
    \qquad \sfq\in\SmoothCore_\alpha.
\end{equation}

We prove convexity on \(\SmoothCore_\alpha\).  Let
\(\sfq_0,\sfq_1\in\SmoothCore_\alpha\) and \(\theta\in[0,1]\), and set
\[
    \sfq_\theta=(1-\theta)\sfq_0+\theta\sfq_1\in\SmoothCore_\alpha.
\]
Since the map \(\gamma\mapsto\Psi^\gamma\) is \(C^2\) by
Theorem~\ref{thm:concavity}\eqref{item:C2}, the one-variable map
\[
    \theta\longmapsto \Psi^{2\dot\sfq^\theta}(0,0)
\]
is \(C^2\).  Its second derivative is
\[
    D^2\Psi^{2\dot\sfq^\theta}
    \bigl[2(\dot\sfq_1-\dot\sfq_0),2(\dot\sfq_1-\dot\sfq_0)\bigr](0,0),
\]
which is nonpositive by Theorem~\ref{thm:concavity}\eqref{item:concavity}.
Hence
\[
    \theta\longmapsto \Psi^{2\dot\sfq^\theta}(0,0)
\]
is concave.  Using \eqref{eq:chen-initial-smooth-representation} and the
affineness of \(\sfq\mapsto\sfq(1)\) on \(\SmoothCore_\alpha\), we obtain
\begin{equation}
    \psi(\sfq_\theta)
    \le
    (1-\theta)\psi(\sfq_0)
    +
    \theta\psi(\sfq_1).
    \label{eq:convexity-smooth-core}
\end{equation}
Thus \(\psi\) is convex on \(\SmoothCore_\alpha\).

By Lemma~\ref{lem:PathSpace-alpha-density} and Lemma~\ref{lem:smooth-core-L1-stability},
$\psi$ admits a unique $L^1$-continuous extension to $\PathSpace_1$.
By the same \(L^1\)-stability estimate, applied with generalized inverses, this extension
agrees on \(Q_\infty\) with the functional \(\psi(\sfq)=\sfq(1)-\Phi^{2\sfq}(0,0)\) defined in
\eqref{eq:def-psi}.
We still denote this extension
by \(\psi\).

It remains to prove convexity of the extension.
Let \(\sfq_0,\sfq_1\in\PathSpace_1\) and \(\theta\in[0,1]\).  By
Lemma~\ref{lem:PathSpace-alpha-density}, choose
\(\sfq_{i,n}\in\SmoothCore_\alpha\) such that
\[
    \sfq_{i,n}\to\sfq_i
    \qquad\text{in }L^1,
    \qquad i=0,1.
\]
Set
\[
    \sfq_{\theta,n}
    =
    (1-\theta)\sfq_{0,n}+\theta\sfq_{1,n},
    \qquad
    \sfq_\theta
    =
    (1-\theta)\sfq_0+\theta\sfq_1 .
\]
Then \(\sfq_{\theta,n}\in\SmoothCore_\alpha\) and
\(\sfq_{\theta,n}\to\sfq_\theta\) in \(L^1\).  
Applying \eqref{eq:convexity-smooth-core} yields
\[
    \psi(\sfq_{\theta,n})
    \le
    (1-\theta)\psi(\sfq_{0,n})
    +
    \theta\psi(\sfq_{1,n}).
\]
Letting \(n\to\infty\) and using the \(L^1\)-continuity of \(\psi\), we get
\[
    \psi(\sfq_\theta)
    \le
    (1-\theta)\psi(\sfq_0)
    +
    \theta\psi(\sfq_1).
\]
Thus the \(L^1\)-extension of \(\psi\) is convex on \(\PathSpace_1\), completing
the proof.
\end{proof}

\section*{Acknowledgements}

I warmly thank Jhih-Huang Li at National Taiwan University for his hospitality during my 2023 visit, and Jean-Christophe Mourrat and Victor Issa at ENS Lyon for their hospitality during my 2025 visit. I am grateful to Wai-Kit Lam for bringing Auffinger--Chen's work \cite{AuffingerChen16} to my attention during the former visit. 
I thank Victor Issa for suggesting the problem and for providing several insights, both during and after my visit to ENS Lyon.
I warmly thank Eliran Subag and Justin Ko for helpful advice on the early drafts of the paper.

I gratefully acknowledge support from Eliran Subag's grants: ISF grant No. 2055/21 and ERC grant No. 101165541.

During the preparation of this manuscript, I was made aware that Hong-Bin Chen
was independently investigating the convexity of the initial condition \(\psi\).

\end{document}